\documentclass[12pt,leqno]{amsart} 
\usepackage[backgroundcolor=yellow,colorinlistoftodos]{todonotes}
\usepackage{fourier}
\usepackage{datetime}
\usepackage{graphicx,float,tabularx}
\usepackage[utf8]{inputenc}
\usepackage{enumitem}   
\usepackage{amsmath,amssymb,amsfonts,mathrsfs,mathtools}%,MnSymbol}
\usepackage{hyperref}
\usepackage{enumitem}
\usepackage[all,2cell]{xy}
\xyoption{arrow}
\SelectTips{cm}{10}

\newdateformat{myformat}{\THEDAY\ \monthname[\THEMONTH] \THEYEAR,
  \currenttime}
\hoffset - 2cm
\makeatletter
\renewcommand\@makefnmark{%
  \hbox{\textsuperscript{\normalfont\color{red}\@thefnmark}}}
\makeatother

\makeatletter
\providecommand\@dotsep{5}
\def\listtodoname{List of Todos}
\def\listoftodos{\@starttoc{tdo}\listtodoname}
\makeatother

\newcommand{\bookmarktodo}[2][]{\global\advance\count66 by 1
\footnote{\color{red}\textbf{#2}}%
  \addcontentsline{tdo}{todo}{\protect{
    {\color{red}[\the\count66]}\hskip 1em
    \parbox{14cm}{#2}%
  }}%
}
\catcode`\@=11 %\def\today{January 27, 2023}
\def\@evenfoot{\rule{0pt}{20pt}[\myformat\today] \hfill [{\tt \jobname.tex}]}
\def\@oddfoot{\rule{0pt}{20pt}{[\tt \jobname.tex}]\hfill [\myformat\today]}
\catcode`\@=13

\usepackage[object=vectorian]{pgfornament}

\usepackage{bbold,euscript}

\usepackage{quiver}

\usepackage{tikz}
\usetikzlibrary{arrows.meta}
\usetikzlibrary{decorations.markings}

\newtheorem{theorem}{Theorem}%[subsection]
\newtheorem{corollary}[theorem]{Corollary}

\newtheorem{lemma}[theorem]{Lemma}
\newtheorem{proposition}[theorem]{Proposition}

\theoremstyle{definition}

\newtheorem{remark}[theorem]{Remark}
\newtheorem{definition}[theorem]{Definition}
\def\Coker{{\rm Cok}}
\def\Ker{{\rm Ker}}

\def\T{\mathbb{T}}
\def\Ar{\mathbb{A}}

\def\ee{\varepsilon}

\def\id{{\mathbb 1}}

\def\ttt#1#2{%
  \ifnum#1=0
    \ifnum#2=0
    \else
      \ifnum#2=1
        {\mathbb T}_1%
      \else
        {\mathbb T}_1^{#2}%
      \fi
    \fi
  \else
    \ifnum#2=0
      \ifnum#1=1
        {\mathbb T}_2%
      \else
        {\mathbb T}_2^{#1}%
      \fi
    \else
      \ifnum#1=1
        {\mathbb T}_2%
      \else
        {\mathbb T}_2^{#1}%
      \fi
      \ifnum#2=1
        {\mathbb T}_1%
      \else
        {\mathbb T}_1^{#2}%
      \fi
    \fi
  \fi
}
\def\iii#1#2{%
  \ifnum#1=0
    \ifnum#2=0
    \else
      \ifnum#2=1
        {\mathbb T}_2%
      \else
        {\mathbb T}_2^{#2}%
      \fi
    \fi
  \else
    \ifnum#2=0
      \ifnum#1=1
        {\mathbb T}_1%
      \else
        {\mathbb T}_1^{#1}%
      \fi
    \else
      \ifnum#1=1
        {\mathbb T}_1%
      \else
        {\mathbb T}_1^{#1}%
      \fi
      \ifnum#2=1
        {\mathbb T}_2%
      \else
        {\mathbb T}_2^{#2}%
      \fi
    \fi
  \fi
}
\def\TT#1#2{\T_{#1}^{#2}}

\newcommand{\C}{{\tt C}}
\newcommand{\K}{{\tt K}}
\newcommand{\D}{\tt{D}}
\newcommand{\Cat}{{\tt Cat}}

\newcommand{\uu}{\mathbb{1}}

\newcommand{\NN}{\mathbb{N}}

\newcommand{\eK}{\alpha_2}%{\EuScript{K}}
\newcommand{\eC}{\alpha_1}%{\EuScript{C}}

\newcommand{\x}{\times}

\def\dvedve{\mathbf{2}}

\newcommand{\oAr}{\overline{\Ar}}

\newcommand{\Vr}{\text{\raisebox{\depth}{\scalebox{1}[-1]{$\Ar$}}}}
\newcommand{\uVr}{\underline{\text{\raisebox{\depth}{\scalebox{1}[-1]{$\Ar$}}}}}

\DeclareMathOperator{\coker}{coker}

\renewcommand{\epsilon}{\varepsilon}
\renewcommand{\phi}{\varphi}
\renewcommand{\rho}{\varrho}

\let\pf\proof
\let\epf\endproof

\title[Lax Distributivity and Abelian Categories]
{Lax Distributivity and a Characterization of Abelian Categories}
\author[M.\ Markl, D.\ Trnka]{Martin Markl and Dominik Trnka}
\address{M.M.: Institute of Mathematics, The Czech
  Academy of Sciences, {\v Z}itn{\'a} 25, 110 00 Praha, The Czech Republic}

\address{D.T.: Institute of Mathematics, University of Technology,
  Technick\'a 2896, 616 69 Brno, The Czech Republic}
\curraddr{The same as M.M.}

\catcode`\@=11
\email{markl@math.cas.cz; trnka@fme.vutbr.cz}
\@namedef{subjclassname@2010}{%
  \textup{2020} Mathematics Subject Classification}
\catcode`\@=13


\begin{document}
\baselineskip 17pt plus 2pt minus 1pt

\begin{abstract}
 We show that abelian categories can be characterized as structures
 consisting of a~colax algebra and a lax algebra
 connected by a lax mixed rewriting rule. To this end we develop 
a~theory of  lax rewriting rules for pairs of lax-lax and
 colax-lax algebras over $2$-monads.
\end{abstract}

\thanks{M.M. supported by  RVO: 67985840.}

\maketitle

\tableofcontents

\section*{Introduction}

A recurring theme in category theory is that properties of
mathematical objects can often be replaced by algebraic structures. Familiar
examples include finite products, limits, fibrations, or factorization systems,
all of which can be described as algebras for suitable
2-monads. This raises the question whether more subtle categorical
notions, such as exactness properties, can also be understood
algebraically.

The axioms of abelian categories involve
the existence and interaction of kernels, cokernels and binary
products. Our aim is to replace these property-based axioms by explicit
algebraic structures together with a coherent compatibility
conditions. Our main result achieves precisely this.

The idea of algebraic characterization of abelian categories emerged
from the preprint~\cite{bivar} of the first author. In
\cite{pasticio} we proved that a pointed category has kernels if
and only if it is a~normalized lax algebra for
the arrow $2$-monad. Then, dually, a pointed category has cokernels
if and only if it carries the structure of a normalized colax algebra
for the `complementary' arrow $2$-monad. Here we study the
interaction of these two algebraic structures via a notion of lax 
distributivity of monad algebras, which we introduce and 
call a \textit{lax rewriting rule},
and prove~that

\begin{center}
\textit{abelian categories are precisely additive categories equipped with
normalized colax and lax algebra structures for two arrow category
$2$-monads satisfying an appropriate lax rewriting rule.}
\end{center}

We now briefly describe the algebraic framework underlying this
result.  The theory of monads and their algebras is classical and
provides a uniform language for describing algebraic and
categorical structures.
Beyond individual algebraic theories, 
Beck's distributive law of monads~\cite{Beck} describes how different algebraic
structures interact and combine into a new one.  If $\T_1$ and $\T_2$
are monads on a category $\C$, a distributive law
\[
\lambda \colon \T_1\T_2 \Longrightarrow \T_2\T_1
\]
allows the composite $\T_{21}:=\T_2\T_1$ to inherit a monad
structure. Its algebras are
objects carrying both a $\T_1$-algebra and a
$\T_2$-algebra structure satisfying a compatibility condition, called
the~\textit{rewriting rule}.
A~prototypical example is that of rings; if $\T_1$ is the free
monoid monad and $\T_2$ the free abelian group monad on
sets, the algebras for the composite monad~$\T_{21}$ are rings.

An interesting natural distributive law arises 
in the $2$-categorical setting. Let
$\Ar$ denote the $2$-monad sending a category $\C$ to its arrow
category. Then $\Ar^2 \C $ is
the category of commutative squares
\hphantom{.}\vskip -3em
\[
\xymatrix{a \ar[r]^{h_0} \ar[d]_{f}  & c \ar[d]^{g}
\\
b  \ar[r]^{h_1} & d.
}
\]
and $\Ar$ distributes over itself via the involution that 
transposes such a square along its main diagonal
\hphantom{.}\vskip -3em 
\[
\xymatrix{a \ar[r]^{h_0} \ar[d]_{f}  & c \ar[d]^{g}
\\
b  \ar[r]^{h_1} & d
}
\raisebox{-1.8em}{ $\ \xmapsto{\lambda} \ $}
\xymatrix{a \ar[r]^{f} \ar[d]_{h_0}  & b \ar[d]^{h_1}
\\
c  \ar[r]^{g} & d\ .} 
\]
Although elementary, this distributive law turns out to govern 
our algebraic characterization of abelian categories.

While distributive laws of monads are well understood, much less
is known in the $2$-ca\-te\-gorical setting about distributivity at the
level of their algebras. Here 
new phenomena arise from the coexistence of colax and
lax algebra structures. To distinguish these two levels, we reserve
the term \textit{distributive law} for the classical
monads, and use the term \textit{rewriting rule}\/ for the
corresponding compatibility between their algebras. Throughout this
paper, the distributive laws of monads are strict and
invertible, whereas the rewriting rules are
allowed to be~lax.

As in the strict case, a pair of lax $\T_i$-algebras
$(A,\alpha_i,\phi_i)$, $i=1,2$, satisfying a lax rewriting rule,
determines a lax algebra for the composite monad
$\T_{21}$. The converse, however, fails in general:
a lax algebra for the
composite monad need not decompose into a pair of  lax
algebras. More importantly, it is natural to consider also rewriting
rules relating a colax algebra to a lax algebra. In this situation
there is no evident notion of a composite algebra, yet the
compatibility remains meaningful and captures
essential structural information. Finally, we express this
compatibility by lifting one monad to the $2$-category of
(co)lax algebras for the other.

The present article is the 
second in a series begun in~\cite{pasticio}. The final paper will
characterize abelian categories in terms of the double d\'ecalage of
the categorical nerve.

\vskip .5em
\noindent 
{\bf Plan of the paper.}
{\em Section~\ref{Jarka se vrati ze Sardinie az ve stredu.}\/}
recalls the classical distributive laws of monads and
rewriting rules between their algebras. {\em Section~\ref{section:lax_lax}}\
develops the theory of lax rewriting rules for pairs of lax algebras, while
{\em Section~\ref{Zitra jedem na chalupu zase kvuli te strese.}}\ 
extends it to mixed rewriting rules relating colax and lax
algebras. {\em Section~\ref{Moc ji to tam nesedlo.}}\ reviews our monadic
approach to kernels and cokernels. 
In {\em Section~\ref{section:ker_coker}}\ we study the
lax rewriting rule arising from
kernels and cokernels. {\em Section~\ref{Pres vikend to htozne foukalo.}}\
presents our characterization
of abelian categories. The three sections of the {\em Appendix}\ recall
colax and lax algebras and their morphisms,
and contain proofs too technical to include in the main text. The main
results of the article are  Theorem~\ref{Ozivil jsem Jarcino
  auto.}, Proposition~\ref{proposition:pre-abelian} and
Theorem~\ref{Dnes jsme s Jarkou pekne zmokli.}.

\vskip .5em
\noindent 
{\bf Conventions.}
All $2$-categories featured in this article are strict, and all their (co)lax
algebras are normalized, meaning that all unit laws are
satisfied strictly. The dot ``$\cdot$'' denotes the
horizontal composition, while the box ``$\Box$'' denotes 
the vertical composition in a $2$-category.
The symbol $\id$, with or without
a subscript, denotes---depending on the context---the identity
endomorphism, functor, or transformation.

\section{Distributive laws and rewriting rules}
\label{Jarka se vrati ze Sardinie az ve stredu.}

In this brief section, we recall some classical material on
distributive laws between monads and rewriting rules for the
associated algebras, following~\cite[Section~3]{fox-markl:ContM97}. In
the subsequent sections, we generalize these results to various
(co)lax settings.

Assume that
$\T_i = (\T_i,\mu_i,\eta_i)$ are
monads on a category $\C$, with multiplications  $\mu_i : \T_i^2
\Rightarrow  \T_i$ and units $\eta_i : \id_\C \Rightarrow \T_i$, $i=1,2$.
Given a natural transformation $\lambda : \T_1   \T_2 \Rightarrow \T_2
\T_1$, we denote by
\[
\Lambda : \T^m_1\T^n_2 \Longrightarrow  \T^n_2\T^m_1, \ m,n \geq 0,
\]
the natural transformation built from repeated application of
$\lambda$.

\begin{definition}[Definition 3.1 of \cite{fox-markl:ContM97}]
\label{Pujde dnes k ochotnikum?}
A {\/\em distributive law\/} of $\T_1$  over $\T_2$ is a natural
transformation   $\lambda :  \T_1   \T_2 \Rightarrow \T_2 \T_1$
satisfying
\begin{subequations}
\label{Jarka se vratila z Kolumbie.}
\begin{align}
\Lambda \cdot \mu_1\T_2 = \T_2 \mu_1 \cdot \Lambda, 
\hskip 2em
&\Lambda \cdot \eta_1\T_2 = \T_2 \eta_1, \ \hbox{ and}
\\
\Lambda \cdot \T_1\mu_2 = \mu_2 \T_1 \cdot \Lambda,
\hskip 2em
& \Lambda \cdot \T_1\eta_2 = \eta_2 \T_1.
  \end{align}
\end{subequations}
\end{definition}

In the situation of Definition~\ref{Pujde dnes k ochotnikum?}, one can form
the composite monad $\T_{21} = (\T_{21},\mu_{21},\eta_{21})$
defined by
\[
\T_{21} :=  \T_2\T_1, \ \mu_{21} := \mu_2 \T_1 \cdot \T^2_2 \mu_1 \cdot \Lambda, \
\hbox {and } \eta_{21} := \eta_2\T_1 \cdot \eta_1. 
\]
Throughout this article, we assume that $\lambda$, and hence also
$\Lambda$, is an isomorphism. This level of generality is sufficient
for our applications and significantly simplifies the notation in
several places.
We may therefore omit the $\Lambda$'s from the notation and freely
commute $\T_1$ and $\T_2$ past one another. Consequently,
equations~(\ref{Jarka se vratila z Kolumbie.}) can be written as
\begin{align}
\label{Jdu si koupit podlozky.}
\mu_1\T_2 = \T_2 \mu_1,
\hskip .5em
\eta_1\T_2 = \T_2 \eta_1, \hskip .5em   
\T_1\mu_2 = \mu_2 \T_1, \
\hbox { and } \
&\T_1\eta_2 = \eta_2 \T_1.
\end{align}
  Note that the inverse $\lambda^{-1}$ of the invertible distributive law $\lambda$ is also a distributive law.
It is simple to verify that,  for monadic 
algebras $\alpha_1 : \T_1 A \to A$ and  
$\alpha_2 : \T_2 A \to A$ satisfying {\/\em the rewriting rule\/}
\begin{equation}
\label{Jaruska ma uzeh.}
\alpha_1 \cdot \T_1\alpha_2 = \alpha_2 \cdot \T_2 \alpha_1 \cdot \lambda ,
\end{equation}
the formula 
\begin{equation}
\label{Dnes divadlo vynecha.}
\alpha := \alpha_2 \cdot \T_2\alpha_1  : \T_{21} A = \T_2 \T_1 A
\longrightarrow A
\end{equation}
defines a monadic
$\T_{21}$-algebra $\alpha : \T_{21} A \to A$.
Further, 
the formula
\begin{equation}\label{equation:lifted_T_2}
\widetilde{\T}_2(A, \alpha_1)=(\T_2A,\T_2\alpha_1\cdot\lambda)
\end{equation}
defines a $\T_1$-algebra, and the monad ${\T}_2$ lifts to the monad
$\widetilde{\T}_2$ on the category of $\T_1$-algebras.  Symmetrically,
the formula
\begin{equation}\label{equation:lifted_T_1}
\widetilde{\T}_1(A, \alpha_2)=(\T_1 A,\T_1\alpha_2\cdot\lambda^{-1})
\end{equation}
defines a $\T_2$-algebra and $\widetilde{\T}_1$ is denotes the lifted
monad on $\T_2$-algebras. Recall the following standard

\begin{proposition}
\label{proposition:known_characterization}
The following data are equivalent.
\begin{enumerate}[topsep=0pt, partopsep=0pt, itemsep=0pt, itemindent=-1em]
\item 
A pair consisting of a $\T_1$-algebra $(A,\alpha_1)$ 
and a $\T_2$-algebra $(A,\alpha_2)$ satisfying the rewriting 
rule~\eqref{Jaruska ma uzeh.}.
\item 
A pair consisting of a $\T_1$-algebra $(A,\alpha_1)$ and a $\T_2$-algebra 
$(A,\alpha_2)$, such that $\alpha_2$ is a morphism of \ 
$\T_1$-algebras and $\alpha_1$ is a morphism of \ $\T_2$-algebras.
\item 
A $\widetilde{\T}_2$-algebra $\big((A,\alpha_1),\alpha_2\big)$, 
where $\widetilde{\T}_2$ is the 
lifted monad on the category of \ $\T_1$-algebras.
\item 
A $\widetilde{\T}_1$-algebra $\big((A,\alpha_2),\alpha_1\big)$,
where $\widetilde{\T}_1$ is the lifted monad on the category of \ $\T_2$-algebras.
\item 
A $\T_{21}$-algebra $(A,\alpha_2\cdot \T_2\alpha_1)$ for the composite monad \ $\T_{21}$.
\end{enumerate}
\end{proposition}

A similar characterization is available for algebra morphisms; a
morphism is a $\T_{21}$-algebra morphism if and only if it
is both a $\T_1$- and $\T_2$-algebra morphism.

\section{Lax rewriting rules of lax algebras}
\label{section:lax_lax}

In this section we generalize  the rewriting rule recalled 
in~\eqref{Jaruska ma uzeh.} to lax algebras 
for a pair of $2$-monads related by a
distributive law. Everything will be
formulated for $2$-monads on a general $2$-category $\K$, although the
main application of the paper---the 
characterization of abelian categories---uses only 
the special cases in which $\K$ is the
$2$-category of categories, or the
$2$-category of coalgebras for a 2-comonad. 
We restrict ourselves to the normalized
case, as this suffices for our purposes.

\begin{definition}
\label{definition:lax_lax_rewriting} 
Assume
that $\T_1$ and $\T_2$ are $2$-monads on a 2-category $\K$, bound by a distributive law $\lambda :  \T_1   \T_2 \Rightarrow \T_2 \T_1$, and let
$\alpha_i : \T_i A \to A$, $i=1,2$, be (normalized) lax 
$\T_i$-algebras with associativity
constraints  
\[
\phi_i\colon \alpha_i \cdot \T_i \alpha_i
\Longrightarrow \alpha_i \cdot \mu_i
\]
in
\[
\xymatrix@R=2.5em@C=2.5em{
\T^2_iA  \ar[r]^{\T_i \alpha_i} \ar[d]_{\mu_i} 
& \T_iA \ar[d]^{\alpha_i} \ar@{=>}[dl]|{\phi_i\rule{.1em}{0pt}}
\\
\T_iA \ar[r]^{\alpha_i} & A,
}
\]
satisfying the standard coherence conditions expressed by the equalities
\begin{subequations}
\begin{equation}
\label{Kdy bude zas letadelko v poradku?}
\xymatrix@R=1.2em@C=2em{
\T_i^3 A  \ar[rr]^{\T_i^2 {\alpha_i}}\ar[dd]_{\mu_i\T_i} && \T_i^2 A  \ar[rd]^{\T_i {\alpha_i}} \ar[dd]_{\mu_i}
\\
&&& \T_i A  \ar[dd]^{\alpha_i} \ar@{=>}[ld]|{\phi_i\rule{.1em}{0pt}}
\\
\T_i^2 A   \ar[rr]^{\T_i {\alpha_i}} \ar[rd]_{\mu_i} && \T_i A 
\ar@{=>}[ld]|{\phi_i\rule{.1em}{0pt}}
\ar[dr]^{\alpha_i}
\\
&\T_i A  \ar[rr]^{\alpha_i} && A
}
\hskip 1em
\raisebox{-3em}{$=$} \hskip 1em
\xymatrix@R=1.2em@C=2em{
\T_i^3 A  \ar[rr]^{\T_i^2 {\alpha_i}}\ar[dd]_{\mu_i\T_i} \ar[rd]^(.6){\T_i\mu_i}   && 
\T_i^2 A  \ar[rd]^{\T_i {\alpha_i}} \ar@{=>}[ld]|(.45){\T_i{\phi_i\rule{.1em}{0pt}}}
\\
&\T_i^2 A   \ar[rr]^{\T_i {\alpha_i}}   \ar[dd]_{\mu_i}&& \T_i A  \ar[dd]^{\alpha_i} 
\ar@{=>}[lldd]|{\phi_i\rule{.1em}{0pt}}
\\
\T_i^2 A  \ar[rd]_{\mu_i} &&
\\
&\T_i A  \ar[rr]^{\alpha_i} && \ A,
}
\end{equation}
\begin{align*}
    \phi\cdot\T_i\eta_i&=\id_{\alpha_i}, \ \hbox { and}
\\
    \phi\cdot\eta_i\T_i&=\id_{\alpha_i},
\end{align*}
of $2$-cells, $i=1,2$. 
 A lax rewriting rule of a lax $\T_1$-algebra $(A, \alpha_1,\phi_1)$ over a lax $\T_2$-algebra $(A, \alpha_2,\phi_2)$ is a 2-cell $\psi$ filling the pentagon
\begin{equation}\label{Vcera jsme byli na ujetem filmu.}
\xymatrix@C=2.7em@R=1.2em{
\T_1\T_2 A \ar[r]^{\T_1\alpha_2} \ar[d]_\lambda^\cong    
& \T_1 A \ar[dd]^{\alpha_1}   \ar@{=>}[ddl]|{\psi\rule{.1em}{0pt}}
\\
\T_2 \T_1 A
\ar[d]_{\T_2\alpha_1} 
\\
\T_2A   \ar[r]^{\alpha_2}    &A 
}
\end{equation}
\end{subequations}
such that the following four equalities of \/ $2$-cells hold:
\begin{subequations}\label{V halal bistru je vzdy dobre jidlo.}
\begin{equation}
\label{B}
%VRCHOLY
\def\OOO{{\ttt12}A}      
\def\OIO{{\ttt02}A}     
\def\OII{{\ttt01}A}      
\def\III{A}        
\def\IOO{{\ttt11}A}      
\def\IOI{{\ttt01}A}    
\def\IIO{{\ttt01}A}         
\def\OOI{{\ttt11}A}     
%HRANY
\def\OAO{{\ttt02}\alpha_2}  
\def\OIA{{\ttt01}\alpha_1}  
\def\AII{\alpha{_1}} 
\def\AOO{{\ttt10}\mu_1}  
\def\IOA{{\ttt10}\alpha{_1}} 
\def\IAI{\alpha_2}  
\def\IAO{{\ttt01}\alpha{_2}}
\def\IIA{\alpha{_1}}
\def\AIO{\mu_1}
\def\OAI{{\ttt01}\alpha{_2}} 
\def\AOI{{\ttt10}\alpha_1} 
\def\OOA{{\ttt11}\alpha{_1}}
%STENY
\def\BB{}
\def\NN{{\ttt01}\psi}
\def\WW{{\ttt10}\phi_1}
\def\FF{\psi}
\def\SS{\psi}
\def\EE{\phi_1}
  \xymatrix@R=1.2em@C=2.2em{
    \OOO \ar[rr]^{\OAO}\ar[dd]_{\AOO}
    &&\OIO \ar[dd]^{\AIO} \ar[rd]^{\OIA}
    %\ar@{=>}[lldd]|{\BB}
    \\
    &&& \OII  \ar[dd]^{\AII}
   \ar@{=>}[ld]|{\EE\rule{.1em}{0pt}}
    \\
    \IOO\ar[rr]^{\IAO} \ar[rd]_{\IOA} && \IIO
    \ar@{=>}[ld]|{\SS\rule{.1em}{0pt}}
    \ar[dr]^{\IIA}
    \\
    &\IOI \ar[rr]^{\IAI} && \III,
  }
  \hskip .1em
  \raisebox{-3em}{$=$} \hskip .1em
  \xymatrix@R=1.2em@C=2em{
    \OOO \ar[rr]^{\OAO}\ar[dd]_{\AOO} 
    \ar[rd]^(.6){\OOA}   && 
    \OIO \ar[rd]^{\OIA} 
\ar@{=>}[ld]|\NN
    \\
    &\OOI \ar[rr]^{\OAI}   \ar[dd]_{\AOI}
    \ar@{=>}[ld]|{\WW\rule{.1em}{0pt}}
    && \OII 
    \ar[dd]^{\AII} 
    \ar@{=>}[lldd]|{\FF\rule{.1em}{0pt}}
    \\
    \IOO\ar[rd]_{\IOA} &&
    \\
    &\IOI \ar[rr]^{\IAI} &&\III,
  }
\end{equation}
\def\OOO{{\ttt21 A}}      \def\OAO{{\ttt11}\alpha{_2}}  \def\BB{{\ttt10}\psi{}}
\def\OIO{{\ttt11 A}}      \def\OIA{{\ttt01}\alpha{_2}}  \def\EE{\psi} 
\def\OII{{\ttt01 A}}      \def\AII{\alpha{_1}}  \def\SS{\phi_2}
\def\IIO{{\ttt10} A}      \def\AOO{{\ttt 20}\alpha{_1}}  \def\NN{{\ttt01}\phi_2}
\def\III{A}      \def\IAI{\alpha{_2}}  \def\WW{\psi}
\def\IOO{{\ttt20}A}      \def\IOA{\mu{_2}}  \def\FF{\psi}
\def\IOI{{\ttt10}A}      \def\IIA{\alpha{_2}}
\def\OOI{{\ttt11}A}      \def\AIO{{\ttt10}\alpha{_1}}
\def\IAO{{\ttt10}\alpha{_2}}
\def\OOA{{\ttt01}\mu_2}
\def\AOI{{\ttt10}\alpha{_1}} 
\def\OAI{{\ttt01}\alpha{_2}}
\begin{equation}
\label{D}
  \xymatrix@R=1.2em@C=2em{
    \OOO \ar[rr]^{\OAO}\ar[dd]_{\AOO}
    &&\OIO \ar[dd]^{\AIO} \ar[rd]^{\OIA}
    \ar@{=>}[lldd]|{\BB}
    \\
    &&& \OII  \ar[dd]^{\AII}
    \ar@{=>}[ld]|{\EE\rule{.1em}{0pt}}
    \\
    \IOO\ar[rr]^{\IAO} \ar[rd]_{\IOA} && \IIO
    \ar@{=>}[ld]|{\SS\rule{.1em}{0pt}}
    \ar[dr]^{\IIA}
    \\
    &\IOI \ar[rr]^{\IAI} && \III,
  }
  \hskip .1em
  \raisebox{-3em}{$=$} \hskip .1em
  \xymatrix@R=1.2em@C=1.5em{
    \OOO \ar[rr]^{\OAO}\ar[dd]_{\AOO} 
    \ar[rd]^(.6){\OOA}   && 
    \OIO \ar[rd]^{\OIA} \ar@{=>}[ld]|\NN
    \\
    &\OOI \ar[rr]^{\OAI}   \ar[dd]_{\AOI}
   % \ar@{=>}[ld]|{\WW\rule{.1em}{0pt}}
    && \OII 
    \ar[dd]^{\AII} 
    \ar@{=>}[lldd]|{\FF\rule{.1em}{0pt}}
    \\
    \IOO\ar[rd]_{\IOA} &&
    \\
    &\IOI \ar[rr]^{\IAI} && \III,
  }
\end{equation}
\begin{align}
%\begin{equation}
\label{equation:lax_distr_condition_3}
\psi\cdot\T_2\eta_1(A)&=\id_{\alpha_2}, \ \hbox { and}
%\end{equation}
\\
%\begin{equation}
\label{equation:lax_distr_condition_4}
\psi\cdot\T_1\eta_2(A)&=\id_{\alpha_1}.
%\end{equation}
\end{align}
\end{subequations}
\end{definition}

\begin{remark}
Later, in Section~\ref{Zitra jedem na chalupu zase kvuli te strese.}, 
we describe lax rewriting rules for a colax algebra over a lax
algebra. Our conventions are designed for the situation where
$\alpha_1 : \T_1 A \to A$ is the cokernel functor (a colax
$\T_1$-algebra) and $\alpha_2 : \T_2 A \to A$ is the kernel functor
(a lax $\T_2$-algebra), cf.~Section~\ref{section:ker_coker}. In~the
strict case, assuming as everywhere that $\lambda$ is invertible, instead of~(\ref{Dnes divadlo vynecha.}), we could also
define the composite $\T_{21}$-algebra by
\begin{equation}
\label{Je ctvrtek a uz melu z posledniho.}
\alpha := \alpha_1 \cdot \T_1  \alpha_2 \cdot \lambda^{-1} : 
 \T_{21}A = \T_2 \T_1 A
\longrightarrow A
\end{equation}
since it is, by~(\ref{Jaruska ma uzeh.}), the same
as~(\ref{Dnes divadlo vynecha.}). However, in the lax case the formulas differ
and only one is correct, depending on the direction
of the arrow in~(\ref{Vcera jsme byli na ujetem filmu.}). If we choose
\[
\xymatrix@C=2.7em@R=1.2em{
\T_1\T_2 A \ar[r]^{\T_1\alpha_2} \ar[d]_\lambda^\cong    
& \T_1 A \ar[dd]^{\alpha_1}   \ar@{<=}[ddl]|{\psi\rule{.1em}{0pt}}
\\
\T_2 \T_1 A
\ar[d]_{\T_2\alpha_1} 
\\
\T_2A   \ar[r]^{\alpha_2}    &A 
}
\]
instead, then~(\ref{Je ctvrtek a uz melu z posledniho.}) will be the
correct one. 
  
\end{remark}
%Assume
%that $\T_1$ and $\T_2$, and therefore also $\T$, are $2$-monads,
%$\alpha_1 : \T_1 A \to A$, 
%$\alpha_2 : \T_2 A \to A$ are lax monadic algebras, and~(\ref{Jaruska ma
 % uzeh.}) is replaced by a natural transformation $\psi$ in the diagram 
%\begin{equation}
%\label{Vcera jsme byli na ujetem filmu.}
%\xymatrix@C=2.7em@R=1.2em{
%\T_2\T_1 A \ar[r]^{\T_2\alpha_1} \ar[d]_\lambda^\cong    
%& \T_2 A \ar[dd]^{\alpha_2}   \ar@{=>}[ddl]|{\psi\rule{.1em}{0pt}}
%\\
%\T_1 \T_2 A
%\ar[d]_{\T_1\alpha_2} 
%\\
%\T_1A   \ar[r]^{\alpha_1}    &A .
%}
%\end{equation}
Our goal is to find a lax analogue of 
Proposition~\ref{proposition:known_characterization}. 
Let us first show that the composite
\hbox{$\alpha \colon \T_{21}A \to A$} in~(\ref{Dnes divadlo vynecha.}) with the lax rewriting rule $\psi$ of \eqref{Vcera jsme byli na ujetem filmu.} gives a lax
$\T_{21}$-algebra. 

\begin{theorem}
\label{Ozivil jsem Jarcino auto.}
Let $\psi$ be a lax rewriting rule of lax $\T_i$-algebras $\alpha_i : \T_i A \to A$, $i=1,2$, as in Definition~\ref{definition:lax_lax_rewriting}.
Then the composite
\[
\alpha := \alpha_2 \cdot \T_2\alpha_1  : \T_{21} A = \T_2 \T_1 A
\longrightarrow A
\]
introduced in~\eqref{Dnes divadlo vynecha.} is a lax $\T_{21}$-algebra with 
the associativity constraint
\begin{equation}
    \label{Martin prijel do Brna}
\raisebox{-5em}{$\phi_{21} :=$ \ } 
\xymatrix@R=3.2em@C=3.2em{
{\ttt 22}A  \ar[r]^{{\ttt21}\alpha_1}  
\ar[d]_{{\ttt20}\mu_1}
  & \ar[d]_{{\ttt20}\alpha_1}  
\ar@{=>}[ld]|{{\ttt20} \phi_1\rule{.1em}{0pt}}
 \ar[r]^{{\ttt11}\alpha_2} {\ttt21}A &{\ttt 11}A
 \ar[d]^{{\ttt10}\alpha_1}
 \ar@{=>}[ld]|{{\ttt10} \psi \rule{.2em}{0pt}}
\\
\ar[d]_{\mu_2 {\ttt01}} \ar[r]^{{\ttt20} \alpha_1}
{\ttt 21}A & \ar[r]^{{\ttt 10} \alpha_2}  {\ttt 20}A
\ar[d]_{\mu_2} & {\ttt10}A \ar[d]^{\alpha_2} 
\ar@{=>}[ld] |{\phi_2\rule{.1em}{0pt}}
\\ 
{\ttt 11}A \ar[r]^{{\ttt10}\alpha_1}&  {\ttt10}A  \ar[r]^{\alpha_2}  &A.
}
\end{equation}
\end{theorem}

The proof of Theorem~\ref{Ozivil jsem Jarcino auto.} is is postponed
to Appendix~\ref{Ta nova strecha muj program tohoto tydne velmi zahustuje.}. 

\begin{remark}
In general, there is no hope for the converse statement, i.e., for
reconstructing lax $\T_i$-algebras $\alpha_i$ from an arbitrary lax
$\T_{21}$-algebra $\alpha$, unless $\alpha$ is a pseudoalgebra. A more
detailed treatment of (non-invertible) lax rewriting rules for
(unnormalized) lax algebras will be the subject
of~\cite{Stepan-Trnka}.
\end{remark}
We continue our pursuit of equivalent characterizations of lax
rewriting rules.
%, i.e.~finding an analog of Proposition~\ref{proposition:known_characterization}.
Let us recall the definitions of lax and colax morphisms between lax algebras.

\begin{definition}\label{definition:lax_morphism_in_text} Let $\T$ be a 2-monad.
  Let $(A,\alpha_1,\phi_1)$ and $(B,\alpha_2,\phi_2)$ be lax
  $\T$-algebras. A lax morphism
  $$(f,\theta)\colon (A,\alpha_1,\phi_1)\to (B,\alpha_2,\phi_2)$$ is an
arrow $f\colon A\to B$ and a 2-cell
    $$\theta\colon \alpha_2\cdot \T f  \Rightarrow f\cdot \alpha_1$$
    as in
    \begin{equation}
%\label{equation:lax_lax_morphism_2-cell}
\xymatrix@R=2.5em@C=2.5em{
\T A \ar[r]^{\T f} \ar[d]_{\alpha_1} 
& \T B\ar[d]^{\alpha_2} \ar@{=>}[dl]|{\theta\rule{.1em}{0pt}}
\\
A \ar[r]_{f} & B
}
\end{equation}
satisfying the following two conditions:
\begin{itemize}[topsep=0pt, partopsep=0pt, itemsep=0pt, itemindent=-1.5em]
\item[(i)] there is an equality of 2-cells
\begin{equation}\label{equation:lax-lax-cube}
\xymatrix@R=1.2em@C=2em{
\T^2 A \ar[rr]^{\T^2 f}\ar[dd]^{\mu_A}
&&\T^2  B \ar[dd]^{\mu_B} \ar[rd]^{\T\alpha_2}
\\
&&& \T  A  \ar[dd]^{\alpha_2}
\ar@{=>}[ld]|{\phi_2\rule{.1em}{0pt}}
\\
\T  A \ar[rr]^{\T f} \ar[rd]_{\alpha_1} && \T B
\ar@{=>}[ld]|{\theta\rule{.1em}{0pt}}
\ar[dr]^{\alpha_2}
\\
&A \ar[rr]^{f} && B
}
\hskip .5em
\raisebox{-3em}{$=$} \hskip .5em
\xymatrix@R=1.2em@C=2em{
\T^2  A \ar[rr]^{\T^2 f}\ar[dd]^{\mu_A} 
\ar[rd]^(.6){\T \alpha_1}   && 
\T B  \ar[rd]^{\T \alpha_2} \ar@{=>}[ld]|(.45){\T{\theta\rule{.1em}{0pt}}}
\\
&\T A  \ar[rr]^{\T f}   \ar[dd]_{\alpha_1}
\ar@{=>}[ld]|{\T\phi_1\rule{.1em}{0pt}}
&& \T B 
\ar[dd]^{\alpha_2} 
\ar@{=>}[lldd]|{\theta\rule{.1em}{0pt}}
\\
\T A \ar[rd]_{\alpha_1} &&
\\
& A \ar[rr]^{f} && B,
}
\end{equation}
    \item[(ii)] $\theta\cdot\eta_A=\id_f$.
\end{itemize}
\end{definition}

\begin{definition}\label{definition:colax_morphism_in_text}Let $\T$ be a 2-monad. 
Let $(A,\alpha_1,\phi_1)$ and $(B,\alpha_2,\phi_2)$ be lax
$\T$-algebras. A colax morphism
\[
(f,\theta)\colon (A,\alpha_1,\phi_1)\to (B,\alpha_2,\phi_2)
\] 
is a pair consisting of an arrow $f\colon A\to B$ and a 2-cell
\[
\theta\colon f\cdot \alpha_1 \Rightarrow \alpha_2\cdot \T f
\]
as in
\[
\xymatrix@R=2.5em@C=2.5em{
\T A \ar[r]^{\T f} \ar[d]_{\alpha_1} 
& \T B\ar[d]^{\alpha_2} \ar@{<=}[dl]|{\theta\rule{.1em}{0pt}}
\\
A \ar[r]_{f} & B
}
\]
satisfying the following two conditions:
\begin{itemize}[topsep=0pt, partopsep=0pt, itemsep=0pt, itemindent=-1em]
\item[(i)]
there is an equality of 2-cells
\begin{equation}
    \label{equation:colax-lax-cube}
\xymatrix@R=1.2em@C=1.7em{
\T^2 A \ar[rr]^{\T^2 f}\ar[dd]_{\T \alpha_1} && \T^2 B \ar@{<=}[lldd]|{\T\theta\rule{.1em}{0pt}}\ar[rd]^{\mu_B} \ar[dd]_{\T \alpha_2}
\\
&&& \T B \ar[dd]^{\alpha_2} \ar@{<=}[ld]|{\phi_2\rule{.1em}{0pt}}
\\
\T A  \ar[rr]^{\T f} \ar[rd]_{\alpha_1} && \T B
\ar@{<=}[ld]|{\theta\rule{.1em}{0pt}}
\ar[dr]^{\alpha_2}
\\
&A \ar[rr]^f && B
}
\hskip 1em
\raisebox{-3em}{$=$} \hskip 1em
\xymatrix@R=1.2em@C=1.7em{
\T^2 A \ar[rr]^{\T^2 f}\ar[dd]_{{\T \alpha_1}} \ar[rd]^(.6){\mu_A}   && 
\T^2 B \ar[rd]^{\mu_B} 
\\
&\T A \ar@{<=}[ld]|{\phi_1\rule{.1em}{0pt}} \ar[rr]^{\T f}   \ar[dd]_{\alpha_1}&& \T B \ar[dd]^{\alpha_2} 
\ar@{<=}[lldd]|{\theta\rule{.1em}{0pt}}
\\
\T A \ar[rd]_{\alpha_1} &&
\\
&A \ar[rr]^f && B,
}
\end{equation}
\item[(ii)] 
$\theta\cdot\eta_A=\id_f$.
\end{itemize} 
\end{definition}
Now, the formula
\begin{equation}\label{equation:lifted_T_2_lax_lax}
\widetilde{\T}_2(A, \alpha_1,\phi_1)=(\T_2 A,\T_2\alpha_1\cdot\Lambda, \T_2\phi_1\cdot\Lambda)
\end{equation}
defines a lax $\T_1$-algebra, and similarly, the formula 
\begin{equation}\label{equation:lifted_T_1_lax_lax}
\widetilde{\T}_1(A, \alpha_2,\phi_2)=(\T_1 A,\T_1\alpha_2\cdot\Lambda^{-1}, \T_1\phi_2\cdot\Lambda^{-1})
\end{equation}
gives a lax $\T_2$-algebra. 
Ignoring $\lambda$ for a moment, the 2-cell 
\[
\xymatrix{
\T_1\T_2 A \ar[r]^{\T_1\alpha_2}    
\ar[d]_{\T_2\alpha_1} 
& \T_1 A \ar[d]^{\alpha_1}   \ar@{=>}[dl]|{\psi\rule{.1em}{0pt}}
\\
\T_2A   \ar[r]^{\alpha_2}    &A 
}
\hskip 1em
\raisebox{-1.7em}{$=$} \hskip 1em 
\xymatrix{
\T_1\T_2 A \ar[r]^{\T_2\alpha_1}    
\ar[d]_{\T_1\alpha_2} 
& \T_2 A \ar[d]^{\alpha_2}   \ar@{<=}[dl]|{\psi\rule{.1em}{0pt}}
\\
\T_1A   \ar[r]^{\alpha_1}    &A 
}
\]
can be interpreted simultaneously as a part of lax and colax morphism. 
It turns out that the monad $\T_2$ lifts to the 2-category of lax
$\T_1$-algebras, lax morphisms, and their 2-cells by formula~\eqref{equation:lifted_T_2_lax_lax}, as we show in
Proposition~\ref{proposition:lifted_monad} in the Appendix. See also Definition~\ref{definition:2-cell} of a 2-cell of lax algebra morphisms. Similarly,
the monad $\T_1$ lifts to the 2-category of lax $\T_2$-algebras and
their colax morphisms. Finally, we prove

\begin{proposition}
\label{proposition:lax_lax_characterization}
Let $(A, \alpha_i,\phi_i)$ be a lax $\T_i$-algebra, $i=1,2$, and $\psi$ a 2-cell filling the pentagon
\begin{equation}
\label{V nedeli bude ke 40.}
\xymatrix@C=2.7em@R=1.2em{
\T_1\T_2 A \ar[r]^{\T_1\alpha_2} \ar[d]_\lambda^\cong    
& \T_1 A \ar[dd]^{\alpha_1}   \ar@{=>}[ddl]|{\psi\rule{.1em}{0pt}}
\\
\T_2 \T_1 A
\ar[d]_{\T_2\alpha_1} 
\\
\T_2A   \ar[r]^{\alpha_2}    &A 
}
\end{equation}
The following structures are equivalent.
\begin{enumerate}[topsep=0pt, partopsep=0pt, itemsep=0pt, itemindent=-1em]
\item 
A lax rewriting rule of
Definition~\ref{definition:lax_lax_rewriting}, 
i.e.~the 2-cell $\psi$ 
which satisfies the equalities in~\eqref{V halal bistru je vzdy dobre jidlo.}.
\item A pair of a colax morphism of lax $\T_2$-algebras
\[
\widetilde{\T}_1(A,
\alpha_2,\phi_2)\xrightarrow{(\alpha_1,\psi\cdot \lambda^{-1})}(A, \alpha_2,\phi_2)
\]
and a lax morphism of lax $\T_1$-algebras 
\[\widetilde{\T}_2(A, \alpha_1,\phi_1)\xrightarrow{(\alpha_2,\psi)}(A, \alpha_1,\phi_1)\]
sharing the 2-cell $\psi$.
\item 
    A lax $\widetilde{\T}_2$-algebra
    $\big((A,\alpha_1,\phi_1),(\alpha_2,\psi),\phi_2\big)$, where
    $\widetilde{\T}_2$ is the lifted monad on 
lax $\T_1$-alge\-bras and their lax morphisms.
\item A lax $\widetilde{\T}_1$-algebra
  $\big((A,\alpha_2,\phi_2),(\alpha_1,\psi\cdot \lambda^{-1}),\phi_1\big)$, where
  $\widetilde{\T}_1$ is the lifted monad on lax $\T_2$-alge\-bras and
  their colax morphisms.
\end{enumerate}
\end{proposition}

\pf The proof is an exercise in 2-category theory, namely in expanding
the definitions and collecting the data. The equivalence of (1) and (2) follows from comparing equality \eqref{B} with \eqref{equation:lax-lax-cube}, and \eqref{D} with \eqref{equation:colax-lax-cube}. To see the equivalence of (1) and~(3),
note that condition \eqref{B} expresses that
\hbox{$(\alpha_2,\psi)$} is a lax morphism of lax
$\T_1$-algebras, and~\eqref{D} is exactly condition
\eqref{equation:2-cell} on a lax algebra 2-cell $\phi_2$. The equivalence of (1) and (4) is analogous.
Since we will not use this result we
omit the details. \epf

\section{Mixed lax rewriting rules}
\label{Zitra jedem na chalupu zase kvuli te strese.}

For our applications in Section~\ref{section:ker_coker}, we need to
develop a version of rewriting rules relating lax and colax algebras. Hence,
adapting the material of Section~\ref{section:lax_lax}, we introduce,
in Definition~\ref{definition:lax_rewriting}, 
 a mixed lax
rewriting rule between a colax algebra over a lax algebra
and, in Proposition~\ref{Dnes vecere u indianu.}, provide equivalent
characterizations of this structure.  For the convenience of
the reader, the definitions of (co)lax algebras and their
(co)lax morphisms are recalled in Appendix~\ref{Druhy den v Brne.}.

\begin{definition}
\label{definition:lax_rewriting}
A {\em mixed lax rewriting rule\/} between a colax $\T_1$-algebra
$(A, \alpha_1,\phi_1)$ and a lax $\T_2$-algebra $(A, \alpha_2,\phi_2)$
is a~\hbox{$2$-cell} $\psi$ filling the pentagon
\begin{equation}
\label{Jarka je uz druhy den na Sardinii.}
\xymatrix@C=2.7em@R=1.2em{
\T_1\T_2 A \ar[r]^{\T_1\alpha_2} \ar[d]_\lambda^\cong    
& \T_1 A \ar[dd]^{\alpha_1}   \ar@{=>}[ddl]|{\psi\rule{.1em}{0pt}}
\\
\T_2 \T_1 A
\ar[d]_{\T_2\alpha_1} 
\\
\T_2A   \ar[r]^{\alpha_2}    &A 
}
\end{equation}
such that the follows four equalities of $2$-cells hold: 
\begin{subequations}
\label{Jsem tak prijemne rozplacly.}
\begin{equation}\label{equation:lax_distr_condition_1_mixed}
\xymatrix@R=1.2em@C=1.5em{
\T_2^2\T_1A \ar[rr]^{\T_2^2 \alpha_1}\ar[dd]_{\T_2\T_1\alpha_2} && \T_2^2A \ar@{<=}[lldd]|{\T_2\psi\rule{.1em}{0pt}}\ar[rd]^{\mu_2(A) } \ar[dd]_{\T_2 \alpha_2}
\\
&&& \T_2A \ar[dd]^{\alpha_2} \ar@{<=}[ld]|{\phi_2\rule{.1em}{0pt}}
\\
\T_2\T_1A  \ar[rr]^{\T_2\alpha_1} \ar[rd]_{\T_1\alpha_2} && \T_2A
\ar@{<=}[ld]|{\psi\rule{.1em}{0pt}}
\ar[dr]^{\alpha_2}
\\
&\T_1A \ar[rr]^{\alpha_1} && A
}
\hskip 0.3em
\raisebox{-3em}{$=$} \hskip 0.3em
\xymatrix@R=1.2em@C=1.5em{
\T_2^2\T_1A \ar[rr]^{\T_2^2 \alpha_1}\ar[dd]^{\T_2 \T_1\alpha_2} \ar[rd]^(.6){\mu_2{\T_1A}}   && 
\T_2^2A \ar[rd]^{\mu_2(A)} 
\\
&\T_2\T_1A \ar@{<=}[ld]|{\T_1\phi_2\rule{.1em}{0pt}} \ar[rr]^{\T_2 \alpha_1}   \ar[dd]_{\T_1\alpha_2}&& \T_2A \ar[dd]^{\alpha_2} 
\ar@{<=}[lldd]|{\psi\rule{.1em}{0pt}}
\\
\T_2\T_1A \ar[rd]_{\T_1\alpha_2} &&
\\
&\T_2A \ar[rr]^{\alpha_1} && A,
}
\end{equation}
\begin{equation}
\label{equation:lax_distr_condition_2_mixed}
\xymatrix@R=1.2em@C=1.5em{
\T_1^2\T_2A \ar[rr]^{\T_1^2 \alpha_2}\ar[dd]_{\T_1\T_2\alpha_1} && \T_1^2A \ar@{=>}[lldd]|{\T_1\psi\rule{.1em}{0pt}}\ar[rd]^{\mu_1(A) } \ar[dd]_{\T_1\T_2\alpha_1}
\\
&&& \T_1A \ar[dd]^{\alpha_1} \ar@{=>}[ld]|{\phi_1\rule{.1em}{0pt}}
\\
\T_1\T_2A  \ar[rr]^{\T_1\alpha_2} \ar[rd]_{\T_2\alpha_1} && \T_1A
\ar@{=>}[ld]|{\psi\rule{.1em}{0pt}}
\ar[dr]^{\alpha_1}
\\
&\T_2A \ar[rr]^{\alpha_2} && A
}
\hskip 0.3em
\raisebox{-3em}{$=$} \hskip 0.3em
\xymatrix@R=1.2em@C=1.5em{
\T_1^2\T_2A \ar[rr]^{\T_1^2 \alpha_2}\ar[dd]^{\T_1 \T_2\alpha_1} \ar[rd]^(.6){\mu_1{\T_2A}}   && 
\T_1^2A \ar[rd]^{\mu_1A} 
\\
&\T_1\T_2A \ar@{=>}[ld]|{\T_2\phi_1\rule{.1em}{0pt}} \ar[rr]^{\T_1 \alpha_2}   \ar[dd]_{\T_2\alpha_1}&& \T_1A \ar[dd]^{\alpha_1} 
\ar@{=>}[lldd]|{\psi\rule{.3em}{0pt}}
\\
\T_1\T_2A \ar[rd]_{\T_2\alpha_1} &&
\\
&\T_2A \ar[rr]^{\alpha_2} && A,
}
\end{equation}
\begin{equation}\label{equation:lax_distr_condition_3_mixed}
\psi\cdot\T_2\eta_1(A)=\id_{\alpha_2}, \ \hbox { and}
\end{equation}
\begin{equation}\label{equation:lax_distr_condition_4_mixed}
\psi\cdot\T_1\eta_2(A)=\id_{\alpha_1}.
\end{equation}
\end{subequations}
\end{definition}

The principal difference from the lax--lax rewriting rules treated
in Section~\ref{section:lax_lax} is that,
while a~colax algebra $\alpha_1; \T_1 A \to A$ and a lax algebra  
$\alpha_2; \T_2 A \to A$ still combine to form the composite morphism
$\alpha : \T_{21}A \to A$, the resulting structure is neither lax nor
colax, since the associativity 
constraints $\phi_1$ and $\phi_2$ do not combine into
a $2$-cell analogous to $\phi_{21}$ in Theorem~\ref{Ozivil jsem
  Jarcino auto.}, which therefore has no analog in the colax--lax case.
However, it is easy to verify that
the formula 
\begin{equation}\label{equation:lifted_T_2_colax_lax}
\widetilde{\T}_2(A, \alpha_1,\phi_1)=(\T_2 A,\T_2\alpha_1\cdot\Lambda, \T_2\phi_1\cdot\Lambda)
\end{equation}
defines a colax $\T_1$-algebra, while formula
\begin{equation}\label{equation:lifted_T_1_colax_lax}
\widetilde{\T}_1(A, \alpha_2,\phi_2)=(\T_1A,\T_1\alpha_2\cdot\Lambda^{-1}, \T_1\phi_2\cdot\Lambda^{-1}),
\end{equation}
already given in Section~\ref{section:lax_lax},
defines a lax $\T_2$-algebra.

As in the lax-lax case, the monad $\T_2$ lifts to the 2-category of
colax $\T_1$-algebras and their lax morphisms by formula
\eqref{equation:lifted_T_2_colax_lax}. Similarly, the monad
$\T_1$ lifts to the 2-category of lax $\T_2$-algebras and their colax
morphisms. We have the
following colax-lax  
analog of Proposition~\ref{proposition:lax_lax_characterization}.

\begin{proposition}
\label{Dnes vecere u indianu.}
Let $(A, \alpha_1,\phi_1)$ be a colax $\T_1$-algebra and $(A,
\alpha_2,\phi_2)$ be a lax $\T_2$-algebra, and $\psi$ a 2-cell 
filling the pentagon
\begin{equation}
\xymatrix@C=2.7em@R=1.2em{
\T_1\T_2 A \ar[r]^{\T_1\alpha_2} \ar[d]_\lambda^\cong    
& \T_1 A \ar[dd]^{\alpha_1}   \ar@{=>}[ddl]|{\psi\rule{.1em}{0pt}}
\\
\T_2 \T_1 A
\ar[d]_{\T_2\alpha_1} 
\\
\T_2A   \ar[r]^{\alpha_2}    &A 
}
\end{equation}
Then the following four structures are equivalent.
\begin{enumerate}[topsep=0pt, partopsep=0pt, itemsep=0pt, itemindent=-1em]
\item 
A lax rewriting rule of
Definition~\ref{definition:lax_rewriting}, 
i.e.~the 2-cell $\psi$  which satisfies 
the equalities in~\eqref{Jsem tak prijemne rozplacly.}.
\item 
A pair of a colax morphism of lax $\T_2$-algebras
\[
\widetilde{\T}_1(A,
\alpha_2,\phi_2)\xrightarrow{(\alpha_1,\psi\cdot \lambda^{-1})}(A, \alpha_2,\phi_2)
\]
and a lax morphism of colax $\T_1$-algebras 
\[\widetilde{\T}_2(A, \alpha_1,\phi_1)\xrightarrow{(\alpha_2,\psi)}(A, \alpha_1,\phi_1)\]
sharing the 2-cell $\psi$.
  \item 
    A lax $\widetilde{\T}_2$-algebra $\big((A,\alpha_1,\phi_1),(\alpha_2,\psi),\phi_2\big)$, where $\widetilde{\T}_2$ is the lifted monad on colax $\T_1$-alge\-bras and their lax morphisms.
    \item 
    A colax $\widetilde{\T}_1$-algebra $\big((A,\alpha_2,\phi_2),(\alpha_1,\psi),\phi_1\big)$, where $\widetilde{\T}_1$ is the lifted monad on lax $\T_2$-alge\-bras and their colax morphisms.
\end{enumerate}
\end{proposition}

\begin{proof}
The proof is analogous to the proof of
Proposition~\ref{proposition:lax_lax_characterization} and we omit it. 
\end{proof}

\section{Composite arrow (co)monad}
\label{Moc ji to tam nesedlo.}

In this section, we explain our monadic approach to pointed
categories, kernels, and cokernels. Some of the material is taken
almost verbatim
from~\cite{pasticio}. 
Let $\dvedve$ denote the category with two objects $0$ and $1$, and one
non-identity morphism $a\colon 0\to 1$. The arrow-category 2-functor
\[
\Ar\colon \Cat\longrightarrow \Cat
\]
is given by $\Ar\C:=[\dvedve,\C]$, the category of functors and natural
transformations. Let us describe it in more detail.

The objects of $\Ar (\C)$ are morphisms of $\C$ and the
morphisms of $\Ar (\C)$ are commutative squares
in~$\C$. A~morphism 
\[
S=(h_0,h_1)\colon f\longrightarrow g
\] 
in $\Ar (\C)$ will be
depicted as the square
\begin{equation}
\label{equation:square}  
    \raisebox{-1.8em}{$S\ =$ \ }
\xymatrix{a \ar[r]^{h_0} \ar[d]_{f}  & c \ar[d]^{g}
\\
b  \ar[r]^{h_1} & d.
}
\end{equation}
By convention, the morphisms in $\Ar (\C)$ always go 
from the left edge of a square to the right edge.
For a functor $F\colon \C\to \D$, the value
$\Ar  (F)\colon \Ar (\C) \to \Ar (\mathtt{D})$ 
is given by post-composition with $F$, that is 
\begin{align*}
\Ar  (F)(f\colon a\to b)&=F(f)\colon F(a) \longrightarrow F(b),
\\
\Ar  (F)\big((h_0,h_1)\colon f\to g\big)&=
\big(F(h_0), F(h_1)\big)\colon F(f)\longrightarrow F(g).
\end{align*}
Finally, given a natural 
transformation $\omega\colon F\Rightarrow G$, the transformation 
$\Ar  (\omega)\colon \Ar  (F) \Rightarrow \Ar  (G)$ has components 
\begin{equation}
    \label{Na koncert jede cely sbor vlakem}
\Ar (\omega)_{f}=(\omega_a,\omega_b)\colon F(f)\longrightarrow G(f).
\end{equation}

The category $\dvedve$ has two complementary monoidal structures given by meet and join operations, that is, $(\dvedve,\wedge,1)$ and $(\dvedve,\vee,0)$. 
These monoidal categories are strictly associative and strictly symmetric. %Let $0\colon \jed \to \dvedve$ be the functor that picks the object $0$ and $1\colon \jed \to \dvedve$ picks the object $1$.

\begin{proposition}
\label{V nedeli ma byt snad 40.}
The precomposition with the structure maps of monoidal categories
$(\dvedve,\wedge,1)$ and $(\dvedve,\vee,0)$ induces two 2-comonads
$ (\Vr,\delta^\bullet,\ee^\bullet)$ and
$(\Ar,\delta_\bullet,\ee_\bullet)$ on \/$\Cat$, the category of
categories.
\end{proposition}

\begin{proof}
As functors, $\Ar = \Vr$, but the comonad structures are different. 
For a category~$\C$ and a~morphism 
$f\colon a\to b$ in $\C$, the comonad structures are given by the formulas
    \begin{align*}
    \label{equation:4-square}
    \raisebox{-1.8em}{$\delta^\C(f):=$}&\raisebox{-1.8em}{$f\cdot \wedge=$}
\xymatrix{a\ar[r]^{\id} \ar[d]_{\id}  &a  \ar[d]^{f} 
\\
a\ar[r]^{f} & \, b,
}&
 \raisebox{-1.8em}{$\delta_\C(f):=$}&\raisebox{-1.8em}{$f\cdot \vee=$}
\xymatrix{a\ar[r]^{f} \ar[d]_{f}  &b  \ar[d]^{\id} 
\\
b\ar[r]^{\id} & \, b,
}\\
 \ee^\C(f):=& f(1)= b,&
\ee_\C(f):=&f(0)=a, &
\end{align*}
that use the canonical isomorphism $[\dvedve\x\dvedve\to \C]\cong
\big[\dvedve\to[\dvedve\to \C]\big]=\Ar^2 \C $.
% and $[\jed\to\C]\cong \C$.
Their coassociativity and counitality is implied by the
associativity and unitality of the operations $\wedge$ and $\vee$. 
\end{proof}

A dualization of Definition~\ref{Pujde dnes k ochotnikum?}
gives distributive laws for comonads and rewriting rules for
coalgebras. We can again define the composite comonad for which the
dual version of Proposition~\ref{proposition:known_characterization}
holds. For the arrow-category 2-comonads in Proposition~\ref{V nedeli
  ma byt snad 40.}, we have the following

\begin{proposition}
\label{proposition:composite_comonad}
The precomposition with the canonical symmetry
$\tau\colon\dvedve\x\dvedve\cong \dvedve\x\dvedve$ induces
an invertible, idempotent distributive law
$\lambda \colon \Vr\Ar \Rightarrow \Ar\Vr$. This gives a
composite comonad $\Ar\Vr$ for which an $\Ar\Vr$-coalgebra $(\C,\beta)$
is a triple $(\C,\beta^\C,\beta_\C)$ consisting of a $\Vr$-coalgebra
$\beta^\C\colon \C\to \Vr\C$ and an $\Ar$-coalgebra
$\beta_\C\colon \C\to \Ar\C$, such that the diagram
\begin{equation}
\label{Zitra ma Dominik zkousku.}
\xymatrix{
  &\C  \ar@/_1.3em/[ld]_{\beta^\C}\ar@/^1.3em/[rd]^{\beta_\C}
\\
\Vr\C \ar[d]_{\Vr\beta_C}  && \Ar\C \ar[d]^{\Ar\beta^C}
\\
\Vr\Ar\C  \ar[rr]^\lambda  &&\Ar\Vr\C,
}
\end{equation}
which represents the dual of~\eqref{Jaruska ma uzeh.}, commutes.
\end{proposition}

\pf The axioms of a distributive law for $\lambda$ are implied by
the commutativity of diagrams
\begin{equation*}
\xymatrix@R=2.5em@C=2.5em{
\dvedve \ar[d]_{1\x \id} \ar[dr]^{\id\x 1}&
\\
\dvedve\x\dvedve \ar[r]^{\tau}& \dvedve\x\dvedve
}
\qquad
\xymatrix@R=2.5em@C=2.5em{
\dvedve\x \dvedve\x \dvedve\ar[r]^{\tau \x \id} \ar[d]_{\id \x \wedge}&
\dvedve\x \dvedve\x \dvedve \ar[r]^{\id\x\tau}&
\dvedve\x \dvedve\x \dvedve \ar[d]^{\wedge \x\id}
\\
\dvedve\x\dvedve \ar[rr]^{\tau}&& \dvedve\x\dvedve 
}
\]
\[
\xymatrix@R=2.5em@C=2.5em{
\dvedve \ar[d]_{0\x \id} \ar[dr]^{\id\x 0}&
\\
\dvedve\x\dvedve \ar[r]^{\tau}& \dvedve\x\dvedve
}
\qquad
\xymatrix@R=2.5em@C=2.5em{
\dvedve\x \dvedve\x \dvedve\ar[r]^{\id\x\tau} \ar[d]_{\vee\x\id}&
\dvedve\x \dvedve\x \dvedve \ar[r]^{\tau\x\id}&
\dvedve\x \dvedve\x \dvedve \ar[d]^{\id\x \vee}
\\
\dvedve\x\dvedve \ar[rr]^{\tau}&& \dvedve\x\dvedve. 
}
\end{equation*}
The rest follows from the dual of 
Proposition~\ref{proposition:known_characterization}.\epf

\begin{proposition}
\label{proposition:pointed}
Pointed categories are $\Ar\Vr$-coalgebras with the structure operations
 the unique morphisms 
$\beta^\C(a)=!^a\colon a\to 0$ and $\beta_\C(a)=!_a\colon 0\to a$.
\end{proposition}
\begin{proof}
It is straightforward to check that the zero object gives both
$\Vr$- and $\Ar$-coalgebra structure, and that the
pentagon~(\ref{Zitra ma Dominik zkousku.}
commutes. 
\end{proof}

%Recall from Proposition~\ref{proposition:pointed} that a pointed category $\C$ is viewed as an $\Vr\Ar$-coalgebra $(\C,\beta_\C,\beta^\C)$ with $\beta_\C(a)=!_a$ and $\beta^\C(a)=!^a$. 
In the second part of this section we study induced monads on the category of $\Ar\Vr$-coalgebras.
Recall from \cite[Section~2]{pasticio} that the cofree-forgetful adjunction of the comonad $\Ar$ induces a 2-monad $\oAr$ on the category of $\Ar$-coalgebras. The value $\oAr(\C,\beta_\C) = (\Ar\C,\delta_\C)$ is the cofree coalgebra on $\C$. The multiplication comes from the counit $\mu_{(\C,\beta_\C)} = \Ar\epsilon_\C$ and the component of the unit is the $\Ar$-coalgebra map
$\eta_{(\C,\beta_\C)} = \beta_\C \colon (\C,\beta_\C) \to (\Ar\C,\delta_\C).$ 

\begin{theorem}[Theorem 8 of \cite{pasticio}]
\label{theorem:pasticio_lax_are_kernels}
For a pointed category $\C$, viewed as an $\Ar$-coalgebra
$(\C,\beta_\C)$ with $\beta_\C(a)=!_a$, the terminal map, a
(normalized) lax $\oAr$-algebra structure $(\C,\eK,\phi_2)$ is
equivalent to a choice of a kernel
\[
\kappa_f\colon \eK (f)\to a
\]
for each morphism $f\colon a \to b$ of \ $\C$, such that
$\kappa_{!_a}=\id_a$ and $\kappa_{\id_a}=\ !^a$.
\end{theorem}

Dually, we get the induced monad $\uVr$ on
$\Vr$-coalgebras with $\uVr(\C,\beta^\C) = (\Vr\C,\delta^\C)$ together
with the
following dual of~\cite[Theorem~8]{pasticio}.

\begin{theorem}
For a pointed category \ $\C$, viewed as an $\Vr$-coalgebra
$(\C,\beta^\C)$ with $\beta^\C(a)=!^a$, the initial map, a colax
$\uVr$-algebra structure $(\C,\eC,\phi_1)$ is equivalent to a choice
of a cokernel
\[
\pi_f\colon b\to \eC (f)
\]
for each morphism $f: a \to b$ of \ $\C$, such that $\pi_{!^a}=\id_a$ and
$\pi_{\id_a}=\ !_a$. 
\end{theorem}
 
To arrive at the setting of mixed rewriting rules in
Definition~\ref{definition:lax_rewriting}, we need to extend both
monads $\uVr$ and $\oAr$ to monads $\T_1$ and $\T_2$ that
operate on the same category, namely the
category of $\Ar\Vr$-coalgebras. They will be induced 
by the comonads $\Vr$ and $\Ar$, and by the invertible
distributive law $\lambda\colon \Vr\Ar\rightarrow\Ar\Vr$. For better
orientation, we present the following
scheme
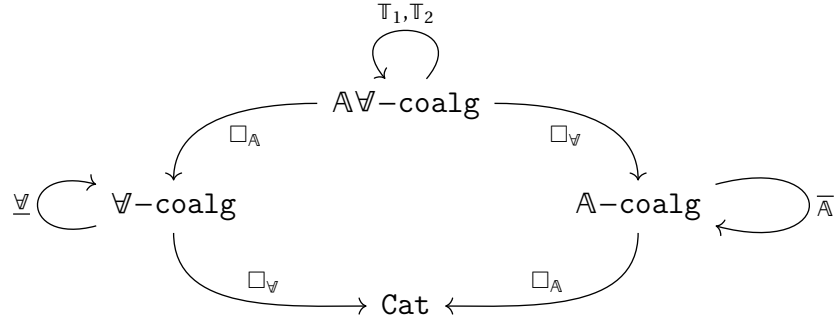
\tikzcdset{arrows={line width=.5pt}}
\[
\begin{tikzcd}
	& {\Ar\Vr\tt{-coalg}} & \\
	{\Vr\tt{-coalg}} && {\Ar\tt{-coalg}} \\
	& \tt{Cat}
	\arrow["{\T_1, \T_2}"', from=1-2, to=1-2, loop ,distance=3em, in=130, out=50]
    \arrow["{\square_\Ar}",from=1-2, to=2-1, in=90,out=180]
	\arrow["{\square_\Vr}"',  in=90, out=0,from=1-2, to=2-3]
	\arrow["\uVr", from=2-1, to=2-1, loop left,looseness=5]
	\arrow["{\square_\Vr}",  in=180, out=270,from=2-1, to=3-2]
	\arrow["\oAr", from=2-3, to=2-3, loop right]
	\arrow["{\square_\Ar}"', in=0,out=270,from=2-3, to=3-2]
\end{tikzcd}
\]
where $\square$'s denote forgetful functors.
\begin{proposition}
    The endofunctor $\T_2$ on $\Ar\Vr$-coalgebras defined by
    $$\T_2(\C,\beta^\C,\beta_\C):=(\Ar\C,\lambda^{-1}\cdot \Ar\beta^\C,\delta_\C)$$
    together with natural transformations 
    $${\mu_2}{(\C,\beta^\C,\beta_\C)}:=\Ar\epsilon_\C\colon \Ar^2 \C \to \Ar\C$$
$${\eta_2}{(\C,\beta^\C,\beta_\C)}:=\beta_\C\colon \C\to \Ar\C$$
is a 2-monad on the category of \/$\Ar\Vr$-coalgebras.
\end{proposition}
\pf First note that $\lambda^{-1}\cdot \Ar\beta^\C$ defines an
$\Vr$-coalgebra, which means the  commutativity of the
diagram
\begin{equation*}
\xymatrix@R=2.5em@C=2.5em{
\Ar\C \ar[d]_{\Ar\beta^\C} \ar[r]^{\Ar\beta^\C}&
\Ar\Vr\C \ar[r]^{\lambda^{-1}} \ar[d]^{\Ar\Vr\beta^\C}
& \Vr\Ar\C \ar[d]^{\Vr\Ar\beta^\C}
\\
\Ar\Vr\C \ar[d]_{\lambda^{-1}} \ar[r]^{\Ar\delta^\C}& \Ar\Vr^2 \C \ar[r]^{\lambda^{-1}\Vr}& \Vr\Ar\Vr\C\ar[d]^{\Vr\lambda^{-1}}
\\
\Vr\Ar\C \ar[rr]^{\delta^{\Ar\C}}&& \Vr^2\Ar\C.
}
\end{equation*}
In that diagram, the upper left square commutes since $\beta^\C$ is an
$\Vr$-coalgebra, the upper right square commutes by naturality, and
the bottom square is the comonad version of~\eqref{Jarka se vratila z
  Kolumbie.}, so the entire diagram commutes as
required.
The diagram
\[
\xymatrix{
&\Ar\C \ar@/_1.3em/[ld]_{\delta_\C} \ar@/^1.3em/[rd]^{\Ar\beta^\C}
\\
\Ar^2 \C  \ar[d]_{\Ar^2 \beta^\C} && \Vr\Ar\C \ar[d]^{\lambda^{-1}}
\ar@{-->}[lld]_{\delta_{\Vr\C}}
\\
\Ar^2\Vr\C \ar[d]_{\Ar \lambda^{-1}}  && \Vr\Ar\C \ar[d]^{\Vr \delta_\C}
\\
\Ar\Vr\Ar\C\ar[rr]^{\lambda^{-1}\Ar} && \Vr\Ar^2 \C 
}
\]
is, as indicated, composed of two diagrams that commute by naturality and by the comonad version of~\eqref{Jarka se vratila z
  Kolumbie.}. Hence it 
also commutes, showing that the pair of coalgebras 
$\lambda^{-1}\cdot \Ar\beta^\C$ and
$\delta_\C$  obeys the rewriting rule and that the object
$\T_2(\C,\beta_\C,\beta^\C)$ is indeed a coalgebra for the composite
comonad $\Ar\Vr$.  The associativity of $\mu_2$ follows from the
naturality of $\epsilon_\bullet$, and it is clear that
\hbox{$\mu_2\cdot\T_2\eta_2=\id$} and $\mu_2\cdot\eta_2\T_2=\id$.  It is
straightforward to check that both ${\mu_2}_{(\C,\beta^\C,\beta_\C)}$
and ${\eta_2}_{(\C,\beta^\C,\beta_\C)}$ are morphisms of
$\Ar\Vr$-coalgebras.  \epf 

The next proposition relates the algebras
of the induced monads $\oAr$ and $\T_2$.

\begin{proposition}   
\label{krasny zapad slunce z 19 patra}
A lax $\T_2$-algebra structure
$\big((\C,\beta^\C,\beta_\C),\alpha_2,\phi_2\big)$ is equivalently a
lax $\oAr$-algebra structure
$\big((\C,\beta_\C),\alpha_2,\phi_2\big)$, for which $\eK$ is also a
$\Vr$-coalgebra map
$$\alpha_2\colon (\Ar\C,\lambda^{-1}\cdot
\Ar\beta^\C)\to(\C,\beta^\C).$$ For a pointed category $\C$, viewed as
an
$\Vr$-coalgebra% with a lax $\oAr$-algebra structure $\big((\C,\beta_\C),\alpha_2,\phi_2\big)$
, any $\alpha_2$ is an \/$\Vr$-coalgebra map.
\end{proposition}
\pf The first part follows from the fact, that a map of $\Ar\Vr$-coalgebras is a map which is both an $\Ar$-coalgebra homomorphism and a $\Vr$-coalgebra homomorphism. For the second part we show commutativity of diagram
\begin{equation}
\xymatrix@C=2.7em@R=1.2em{
\Ar \C \ar[r]^{\alpha_2} \ar[d]_{\Ar\beta^\C}   
&  \C \ar[dd]^{\beta^\C}   
\\
\Ar\Vr\C
\ar[d]_{\lambda^{-1}} 
\\
\Vr\Ar\C   \ar[r]^{\Vr\alpha_2}    &\Vr\C.
}
\end{equation}
For an object $f\in \Ar\C$, we compute 
\[
\beta^\C(\alpha_2(f))=\ !^{\alpha_2(f)}
\]
and
\[\Vr\alpha_2\cdot\lambda^{-1}\cdot\Ar\beta^\C(f)=\Vr\alpha_2\cdot\lambda^{-1}\left(
\raisebox{2em}{\xymatrix{
	0 \ar[d]^{!^a}\ar@{=}[r]& 0 \ar[d]_{!^b}
\\
	a \ar[r]^{f}& b
}}
\right)=\Vr\alpha_2\left(
\raisebox{2em}{\xymatrix{
	0 \ar[r]^{!^a}\ar@{=}[d]& a \ar[d]_{f}
\\
	0 \ar[r]^{!^b}& b
}}
\right)=\left(0\xrightarrow{!^{\alpha_2(f)}}\alpha_2(f)\right).\]
For a morphism $S\in \Ar\C$, the computation is similar.\epf
\begin{corollary}
\label{corollary:T_2_kernels}
For a pointed category \,$\C$, viewed as an $\Ar\Vr$-coalgebra
$(\C,\beta^\C,\beta_\C)$ with $\beta^\C(a)=!^a$ and $\beta_\C(a)=!_a$,
a lax $\T_2$-algebra structure $(\C,\eK,\phi_2)$ is equivalent to a
choice of a kernel
\[
\kappa_f\colon \eK (f)\to a
\]
for each morphism $f\colon a \to b$ of \ $\C$, such that
$\kappa_{!_a}=\id_a$ and $\kappa_{\id_a}=\ !^a$.
\end{corollary}

\pf
A combination of Proposition~\ref{krasny zapad slunce z 19 patra} and Theorem~\ref{theorem:pasticio_lax_are_kernels}. \epf 
Dually, there are induced monads $\uVr$ on $\Vr$-coalgebras and $\T_1$ on $\Ar\Vr$-coalgebras, defined by
$$\T_1(\C,\beta^\C,\beta_\C):=(\Vr\C, \delta^\C,\lambda\cdot \Vr\beta_\C,)$$
    together with natural transformations 
    $${\mu_1}_{(\C,\beta^\C,\beta_\C)}:=\Vr\epsilon^\C\colon \Vr^2 \C \to \Vr\C,$$
$${\eta_1}_{(\C,\beta^\C,\beta_\C)}:=\beta^\C\colon \C\to \Vr\C.$$

\begin{corollary}
\label{corollary:T_1_cokernels}
For a pointed category \,$\C$, viewed as an $\Ar\Vr$-coalgebra
$(\C,\beta^\C,\beta_\C)$ with $\beta^\C(a)=!^a$ and $\beta_\C(a)=!_a$,
a colax $\T_1$-algebra structure $(\C,\eC,\phi_1)$ is equivalent to a
choice of a cokernel 
\[
\pi_f\colon b\rightarrow \eC (f)
\] 
for each morphism $f\colon a \to b$ of\/ $\C$, such that $\pi_{!^a}=\id_a$ and
$\pi_{\id_a}=\ !_a$.
\end{corollary} 

\section{Lax rewriting rule of kernels and cokernels}
\label{section:ker_coker}

In this section we describe the interaction between a colax $\T_1$- and a
lax $\T_2$-algebra structure on a pointed category, where the $\T_i$'s
are the cokernel and kernel monads from the previous section. Assuming
the axiom of choice, every category with kernels and cokernels can be
equipped with a fixed choice of kernel $\kappa_f$ and cokernel $\pi_f$
for each morphism $f$ in the category, such that the normalization
conditions
\[
\kappa_{!_a}=\id_a,\quad \kappa_{\id_a}=!^a,\quad
\pi_{!^a}=\id_a,\quad
\hbox {and} \quad  \pi_{\id_a}=!_a
\] 
hold. We shall always 
assume that kernels and cokernels have been
chosen and normalized, without further mention.                                  
The main result of this section is

\begin{proposition}
\label{proposition:pre-abelian}
A pointed category $\C$ has (chosen, normalized) cokernels and kernels
if and only if it is equipped with a colax
$\T_1$-algebra structure $(\C,\eC,\phi_1)$ and a lax $\T_2$-algebra
structure $(\C,\alpha_2,\phi_2)$. If so, these structures satisfy a
unique lax rewriting rule $\psi$ whose components are morphisms
\[
\psi_S\colon \coker(\eK(S))\xlongrightarrow{\exists!}
\ker(\eC(\lambda(S)))
\]
induced by the universal properties of 
cokernels and kernels, for a square $S \in \Ar^2\C$ representing a
morphism in $\Ar\C$ as in~\eqref{equation:square}.
\end{proposition}

Before the actual proof, we need to recall from \cite[Proposition~6]{pasticio}
that the associativity constraint 
$\phi_2$ of a lax $\T_2$-algebra $(\C,\eK,\phi_2)$ is
uni\-quely determined by the unit
$\nu_2\colon \uu_{\Ar (\C)}\Rightarrow \beta_\C \cdot \ee_\C$ of the
adjunction $(\id,\nu_2)\colon \epsilon_C \dashv
\beta_\C$ by the formula
\begin{equation}\label{equation:phi_via_nu}
\phi_2=\alpha_2\cdot \Ar(\alpha_2\cdot\nu_2).
\end{equation}
The components of $\nu_2$ are
\begin{equation}
\label{Pul pate a uz je tma.}
\nu_2{(f\colon a\to b)}=(\uu_a,!_b)\colon f\longrightarrow \ !_a =\left(
\raisebox{2em}{\xymatrix{
	a \ar[d]^{f}\ar@{=}[r]& a \ar[d]_{!_a}
\\
	b \ar[r]^{!_b}& 0
}}
\right).
\end{equation}
We denote by $\kappa$ the natural transformation 
\begin{equation}\label{equation:iota}
\kappa:=(\alpha_2\cdot\nu_2)\colon
  \alpha_2\Longrightarrow \ee_\C,
  \end{equation}
  whose components $\kappa_f$ realize the kernel maps as
\[
\left(\ker f \to a \right) = \left(\alpha_2(f)\xrightarrow{\kappa_f}
  a\right),
\] 
cf.~equation (20) of \cite{pasticio}.
  
Dually, there is a natural transformation $\nu_1\colon \beta^\C\epsilon^\C \Rightarrow \id_{\Vr\C}$ with components
  \begin{equation}
\label{Pul pate ... dual}
\nu_1{(f\colon a\to b)}=(!^a,\uu_b)\colon !^b\longrightarrow \ f=\left(
\raisebox{2em}{\xymatrix{
	0 \ar[d]^{!^b}\ar[r]^{!^a}& a \ar[d]_{f}
\\
	b \ar@{=}[r]& b
}}
\right).
\end{equation}
For a normalized colax $\Vr$-algebra $(\C,\eC,\phi_1)$, we denote by
$\pi$ the natural transformation
  $$\pi:=(\eC\cdot\nu_1)\colon
  \ee^\C\Longrightarrow \eC,$$
whose components $\pi_f$ realize the cokernel maps as
\[ \left(b\to \coker f \right) = \left(b\xrightarrow{\pi_f} \alpha_1(f)\right).\]
Note that every kernel is a monomorphism, so are the components of
$\kappa$. Likewise, the components of $\pi$ are epimorphisms.

For a square $S\in \Ar\Vr\C$ we draw the following diagram containing the component of the 2-cell~$\psi$ as the top right diagonal map.

\begin{equation}
\label{diagram:abelian_for_S}
\xymatrix@C=3em{
\ker f \ar[r]^{\alpha_2(S)} \ar@{>->}[dd]_{\kappa_f}
& \ker g \ar@{>->}[dd]_{\kappa_g} 
\ar[r]^(.38){\pi_{\alpha_2(S)}} 
\ar@/_1em/[rrd]|(.4){\ \exists !\ }|(.61)\hole   
& \coker(\alpha_2(S))  \ar@/_1em/[ddr] |{\ \exists ! \ }
 \ar[dr]^{\exists ! \psi_S, \exists ! \psi^S} 
\\
&&& \ker ( \alpha_1(\lambda(S))) \ar@{>->}[d]^{\kappa_{\alpha_1(\lambda(S))}}
\\
a \ar[r]^{h_0}\ar[dd]_f & \ar[dd]_g c \ar@{->>}[rr]^{\pi_{h_0}} && 
\coker (h_0) \ar[dd]^{\alpha_1(\lambda(S))}
\\
\ar@{}[r]|S &
\\
b \ar[r]^{h_1}& d  \ar@{->>}[rr]^{\pi_{h_1}} &&  \coker (h_1)
}
\end{equation}
There are, in principle, two morphisms
\[
\psi_S,\psi^S\colon \coker(\eK(S))\longrightarrow
\ker(\eC(\lambda(S))),
\]
where $\psi_S$ is induced by the universal
property of $\ker\big(\eC(\lambda(S))\big)$, and $\psi^S$ is induced
by the universal property of $\coker(\eK(S))$. However, the
top right pentagon of diagram~\eqref{diagram:abelian_for_S}
commutes for both of them. 
Therefore $\psi_S=\psi^S$, since $\pi_{\eK(S)}$ is an epimorphism and
$\kappa_{\eC(\lambda(S))}$ is a monomorphism.

For the proof of Proposition~\ref{proposition:pre-abelian} we need the
following two lemmas.
\begin{lemma}
The two pentagons
\begin{equation}
\label{equation:pentagon_with_e}
\begin{tikzcd}
{\Vr\Ar\C} && {\Ar\Vr\C} \\
{\Vr\C} && {\Ar\C} \\
& \C
\arrow["{\lambda}", from=1-1, to=1-3]
\arrow["{\Vr F}"', from=1-1, to=2-1]
\arrow["{\Ar\epsilon^{\C}}", from=1-3, to=2-3]
\arrow[""{name=0, anchor=center, inner sep=0}, "{\epsilon^\C}"',
from=2-1, curve={height=12pt}, to=3-2]
\arrow["{F}", from=2-3, to=3-2, curve={height=-12pt}]
\arrow["\epsilon^{\Ar\C}"', dashed, dash pattern=on 5pt off 4pt, from=1-1, to=2-3]
\end{tikzcd}
\quad
\begin{tikzcd}
{\Vr\Ar\C} && {\Ar\Vr\C} \\
{\Vr\C} && {\Ar\C} \\
& \C
\arrow["{\lambda}", from=1-1, to=1-3]
\arrow["{\Vr \epsilon_{\C}}"', from=1-1, to=2-1]
\arrow["{\Ar G}", from=1-3, to=2-3]
\arrow[""{name=0, anchor=center, inner sep=0}, "{G}"',
from=2-1, curve={height=12pt}, to=3-2]
\arrow["{\epsilon_\C}", from=2-3, to=3-2, curve={height=-12pt}]
\arrow["\epsilon_{\Vr\C}", dashed, dash pattern=on 5pt off 4pt, from=1-3, to=2-1]
\end{tikzcd}
\end{equation}
commute for arbitrary functors $F\colon
\Ar\C\to \C$ and $G\colon \Vr\C\to \C$.
\end{lemma}
\pf 
Since $\lambda$ is a distributive law it holds $\Ar\epsilon^\C \cdot
\lambda = \epsilon^{\Ar\C}$ and $\Vr\epsilon_\C= \epsilon_{\Vr\C}\cdot
\lambda$, and hence the pentagons commute by naturality of the counits
$\epsilon^\bullet\colon \Vr\Rightarrow \id, \epsilon_\bullet\colon
\Ar\Rightarrow \id$. 
\epf

\begin{lemma}
\label{lemma:lax_distr_concrete}
Let $(\C,\eC,\phi_1)$ be a colax $\T_1$-algebra, $(\C,\eK,\phi_2)$ a
lax $\T_2$-algebra, and $\psi$ a 2-cell filling the
pentagon
\[
\begin{tikzcd}
{\Vr\Ar\C} && {\Ar\Vr\C} \\
{\Vr\C} && {\Ar\C} \ . \\
& \C
\arrow["{\lambda}","{\cong}"', from=1-1, to=1-3]
\arrow["{\Vr\eK}"', from=1-1, to=2-1]
\arrow["{\Ar\eC}", from=1-3, to=2-3]
\arrow[""{name=0, anchor=center, inner sep=0}, "{\eC}"',
from=2-1, curve={height=12pt}, to=3-2]
\arrow["{\eK}", from=2-3, to=3-2, curve={height=-12pt}]
\arrow["\psi", shorten >=5pt, shorten <=5pt,    Rightarrow, from=2-1, to=2-3]
\end{tikzcd}
\] 
\begin{subequations}
Such a 2-cell  is a lax rewriting rule of Definition~\ref{definition:lax_rewriting} if the following four conditions
\label{Martin odjel}
\begin{equation}
    \label{equation:distr1}
   % https://q.uiver.app/#q=WzAsNSxbMiwwLCJBXjJcXEMiXSxbMCwxLCJBXFxDIl0sWzEsMiwiXFxDIl0sWzIsMSwiQVxcQyJdLFswLDAsIkFeMlxcQyJdLFsxLDIsIlxcZUMiLDJdLFszLDIsIlxcZXBzaWxvbl9cXEMiXSxbNCwwLCJsX1xcQyJdLFswLDMsIkFcXGVDIl0sWzQsMSwiQVxcZUsiLDIseyJjdXJ2ZSI6Mn1dLFs0LDEsInNcXGVwc2lsb25fXFxDIiwwLHsiY3VydmUiOi0yfV0sWzksMTAsInNcXGlvdGEiLDAseyJzaG9ydGVuIjp7InNvdXJjZSI6MjAsInRhcmdldCI6MjB9fV1d
\begin{tikzcd}
	{\Vr\Ar\C} && {\Ar\Vr\C} \\
	{\Vr\C} && {\Ar\C} \\
	& \C
	\arrow["{\lambda}", from=1-1, to=1-3]
	\arrow[""{name=0, anchor=center, inner sep=0}, "{\Vr\eK}"', curve={height=12pt}, from=1-1, to=2-1]
	\arrow[""{name=1, anchor=center, inner sep=0},
        "{\Vr\epsilon_\C}", curve={height=-12pt}, from=1-1, to=2-1]
	\arrow["{\Ar\eC}", from=1-3, to=2-3]
	\arrow["\eC"', from=2-1, to=3-2, curve={height=12pt}]
	\arrow["{\epsilon_\C}", from=2-3, to=3-2,  curve={height=-12pt}]
	\arrow["{\Vr\kappa}",shorten >=4pt, shorten <=4pt,  Rightarrow, from=0, to=1]
\end{tikzcd}
=
% https://q.uiver.app/#q=WzAsNSxbMiwwLCJBXjJcXEMiXSxbMCwxLCJBXFxDIl0sWzEsMiwiXFxDIl0sWzIsMSwiQVxcQyJdLFswLDAsIkFeMlxcQyJdLFsxLDIsIlxcZUMiLDJdLFszLDIsIlxcZXBzaWxvbl9cXEMiLDAseyJjdXJ2ZSI6LTJ9XSxbNCwwLCJsX1xcQyJdLFswLDMsIkFcXGVDIl0sWzQsMSwiQVxcZUsiLDJdLFszLDIsIlxcZUsiLDIseyJjdXJ2ZSI6Mn1dLFsxMCw2LCJcXGlvdGEiLDAseyJzaG9ydGVuIjp7InNvdXJjZSI6MjAsInRhcmdldCI6MjB9fV0sWzksOCwiXFxsYW1iZGFfXFxDIiwwLHsic2hvcnRlbiI6eyJzb3VyY2UiOjIwLCJ0YXJnZXQiOjIwfX1dXQ==
\begin{tikzcd}
	{\Vr\Ar\C} && {\Ar\Vr\C} \\
	{\Vr\C} && {\Ar\C} \\
	& \C
	\arrow["{\lambda}", from=1-1, to=1-3]
	\arrow[""{name=0, anchor=center, inner sep=0}, "{\Vr\eK}"', from=1-1, to=2-1]
	\arrow[""{name=1, anchor=center, inner sep=0}, "{\Ar\eC}", from=1-3, to=2-3]
	\arrow["\eC"', from=2-1, to=3-2,  curve={height=12pt}]
	\arrow[""{name=2, anchor=center, inner sep=0}, "{\epsilon_\C}", curve={height=-12pt}, from=2-3, to=3-2]
	\arrow[""{name=3, anchor=center, inner sep=0}, "\eK"',
        curve={height=12pt}, from=2-3, to=3-2]	
\arrow["{\psi}", Rightarrow, shorten >=10pt, shorten <=10pt, from=0, to=1]
	\arrow["\kappa",shorten >=3pt, shorten <=3pt, Rightarrow, from=3, to=2]
\end{tikzcd}
\end{equation}
\begin{equation}
    \label{equation:distr2}
  % https://q.uiver.app/#q=WzAsNSxbMiwwLCJBXjJcXEMiXSxbMCwxLCJBXFxDIl0sWzEsMiwiXFxDIl0sWzIsMSwiQVxcQyJdLFswLDAsIkFeMlxcQyJdLFsxLDIsIlxcZUMiLDIseyJjdXJ2ZSI6Mn1dLFszLDIsIlxcZUsiXSxbNCwwLCJsX1xcQyJdLFswLDMsIkFcXGVwc2lsb25eXFxDIl0sWzQsMSwiQVxcZUsiXSxbMSwyLCJcXGVwc2lsb25eXFxDIiwwLHsiY3VydmUiOi0yfV0sWzUsMTAsIlxccGkiXV0=
\begin{tikzcd}
	{\Vr\Ar\C} && {\Ar\Vr\C} \\
	{\Vr\C} && {\Ar\C} \\
	& \C
	\arrow["{\lambda}", from=1-1, to=1-3]
	\arrow[""{name=2, anchor=center, inner sep=0}, "{\Vr\eK}"', from=1-1, to=2-1]
	\arrow[""{name=3, anchor=center, inner sep=0},"{\Ar\eC}", from=1-3, to=2-3]
	\arrow[""{name=0, anchor=center, inner sep=0}, "\epsilon^\C"',
        curve={height=12pt}, from=2-1, to=3-2]
\arrow[""{name=1, anchor=center, inner sep=0}, "{\eC}", curve={height=-12pt}, from=2-1, to=3-2]
	\arrow["\eK", from=2-3, to=3-2, curve={height=-12pt}]
	\arrow["\pi",shorten >=3pt, shorten <=3pt, Rightarrow, from=0, to=1]
    \arrow["{\psi}", shorten >=10pt, shorten <=10pt, Rightarrow, from=2, to=3]
\end{tikzcd}
= 
\begin{tikzcd}
	{\Vr\Ar\C} && {\Ar\Vr\C} \\
	{\Vr\C} && {\Ar\C} \\
	& \C
	\arrow["{\lambda}", from=1-1, to=1-3]
	\arrow["{\Ar\eK}"', from=1-1, to=2-1]
	\arrow[""{name=0, anchor=center, inner sep=0}, "{\Ar\epsilon^\C}"', curve={height=12pt}, from=1-3, to=2-3]
	\arrow[""{name=1, anchor=center, inner sep=0}, "{\Ar\eC}", curve={height=-12pt}, from=1-3, to=2-3]
	\arrow["\epsilon^\C"', from=2-1, to=3-2, curve={height=12pt}]
	\arrow["\eK", from=2-3, to=3-2,  curve={height=-12pt}]
	\arrow["{\Ar\pi}",shorten >=3pt, shorten <=3pt, Rightarrow, from=0, to=1]
\end{tikzcd}
\end{equation}
\begin{equation}\label{equation:psi_unit_1}
    \psi\cdot \lambda^{-1}\cdot \Ar\beta^\C=\id_{\alpha_2}, \hbox {
      and}
\end{equation}
\begin{equation}\label{equation:psi_unit_2}
   \psi\cdot \Vr\beta_\C=\id_{\alpha_1}
\end{equation} 
\end{subequations}
hold.
\end{lemma}

\begin{proof}
We show how to 
derive equalities in \eqref{Jsem tak prijemne rozplacly.} of
Definition~\ref{definition:lax_rewriting} from the equalities in
\eqref{Martin odjel} by means of an algebraic manipulation. We begin by
recalling the available equalities and facts at our disposal. To simplify the formulas
we ignore all $\lambda$'s and its iterates.
\begin{itemize}[topsep=0pt, partopsep=0pt, itemsep=0pt, itemindent=0pt, leftmargin=2em]
 \item The effect of $\T_1$ and $\T_2$ on the underlying categories is the arrow functor $\T_1\C=\Vr\C$ and $\T_2\C=\Ar\C$,
\item $\T_1\T_2=\T_2\T_1$ (ignoring invertible $\lambda$),
\item $\mu_2\T_1=\T_1\mu_2$ (ignoring invertible $\lambda$),
\item $\mu_2=\T_2\epsilon_\bullet$ (the definition of $\mu_2$),
\item $\phi_2=\alpha_2\cdot\T_2\kappa$ (equality \eqref{equation:phi_via_nu}),
\item $\alpha_1\cdot\T_1\kappa=(\kappa\cdot\T_2\alpha_1)\Box\psi$
      (assumption \eqref{equation:distr1}).
\item For arbitrary $2$-cells $\gamma'\colon f'\Rightarrow g'$, $\gamma''\colon
  f''\Rightarrow g''$ with $\textrm{dom } f''=\textrm{dom }
  g''=\textrm{cod } f'=\textrm{cod } g'$, i.e.
\[
\begin{tikzcd}
	\bullet && \bullet && \bullet 
	\arrow[""{name=0, anchor=center, inner sep=0}, "f'", curve={height=-24pt}, from=1-1, to=1-3]
	\arrow[""{name=1, anchor=center, inner sep=0}, "g'"', curve={height=24pt}, from=1-1, to=1-3]
	\arrow[""{name=2, anchor=center, inner sep=0}, "f''", curve={height=-24pt}, from=1-3, to=1-5]
	\arrow[""{name=3, anchor=center, inner sep=0}, "g''"', curve={height=24pt}, from=1-3, to=1-5]
	\arrow["\gamma'"'{description}, shorten >=2pt, shorten <=2pt,  Rightarrow, from=0, to=1]
	\arrow["\gamma''"'{description},  
shorten >=2pt, shorten <=2pt, Rightarrow, from=2, to=3]
\end{tikzcd}
\]
we have the interchange 
    $$
   (\gamma''\cdot g')\Box(f''\cdot \gamma')= (g''\cdot \gamma')\Box(\gamma''\cdot f').
    $$
\end{itemize}
Using the above, we rewrite 
the left hand side of~\eqref{equation:lax_distr_condition_1_mixed} as
\begin{align*}
&(\phi_2\cdot\T_2^2\alpha_1)\Box(\alpha_2\cdot\T_2\psi)\Box(\psi\cdot\T_2\T_1\alpha_2)=
(\alpha_2\cdot\T_2\kappa\cdot\T_2^2\alpha_1)\Box(\alpha_2\cdot\T_2\psi)\Box(\psi\cdot\T_2\T_1\alpha_2)=\\
&\left(\alpha_2\cdot\T_2(\kappa\cdot\T_2\alpha_1\Box\psi)\right)\Box(\psi\cdot\T_2\T_1\alpha_2)
=\big(\alpha_2\cdot\T_2(\alpha_1\cdot\T_1\kappa)\big)\Box(\psi\cdot\T_2\T_1\alpha_2).
\end{align*}
Likewise, the right hand side
of~\eqref{equation:lax_distr_condition_1_mixed} 
rewrites as
\begin{align*}
&(\psi\cdot\mu_2\T_1)\Box(\alpha_1\cdot\T_1\phi_2)=
(\psi\cdot\T_1\mu_2)\Box(\alpha_1\cdot\T_1\phi_2)=
\\
&(\psi\cdot\T_1\mu_2)\Box \big((\alpha_1\cdot\T_1\alpha_2)\cdot\T_1\T_2\kappa\big)=
\big((\alpha_2\cdot\T_2\alpha_1)\cdot\T_1\T_2\kappa\big)\Box(\psi\cdot\T_1\T_2\alpha_2)=
\\
&\big((\alpha_2\cdot\T_2\alpha_1)\cdot\T_2\T_1\kappa\big)\Box(\psi\cdot\T_2\T_1\alpha_2)=\big(\alpha_2\cdot\T_2(\alpha_1\cdot\T_1\kappa)\big)\Box(\psi\cdot\T_2\T_1\alpha_2).
\end{align*}
We conclude that the two sides of~\eqref{equation:lax_distr_condition_1_mixed}
are equal. 
The equality of the terms in the second line of the 
second computation follows from the interchange of 
% https://q.uiver.app/#q=WzAsMyxbMCwwLCJcXGJ1bGxldCJdLFsyLDAsIlxcYnVsbGV0Il0sWzQsMCwiXFxidWxsZXQiXSxbMCwxLCJcXFRfMVxcVF8yXFxhbHBoYV8yIiwwLHsiY3VydmUiOi00fV0sWzAsMSwiXFxUXzFcXG11XzIiLDIseyJjdXJ2ZSI6NH1dLFsxLDIsIlxcYWxwaGFfMlxcY2RvdFxcVF8yXFxhbHBoYV8xIiwwLHsiY3VydmUiOi00fV0sWzEsMiwiXFxhbHBoYV8xXFxjZG90XFxUXzFcXGFscGhhXzIiLDIseyJjdXJ2ZSI6NH1dLFszLDQsIlxcVF8xXFxUXzJcXGlvdGEiLDIseyJzaG9ydGVuIjp7InNvdXJjZSI6MjAsInRhcmdldCI6MjB9fV0sWzUsNiwiXFxwc2kiLDIseyJzaG9ydGVuIjp7InNvdXJjZSI6MjAsInRhcmdldCI6MjB9fV1d
\[\begin{tikzcd}
	\bullet && \bullet && \bullet
	\arrow[""{name=0, anchor=center, inner sep=0}, "{\T_1\T_2\alpha_2}", curve={height=-24pt}, from=1-1, to=1-3]
	\arrow[""{name=1, anchor=center, inner sep=0}, "{\T_1\mu_2}"', curve={height=24pt}, from=1-1, to=1-3]
	\arrow[""{name=2, anchor=center, inner sep=0}, "{\alpha_1\cdot\T_1\alpha_2}", curve={height=-24pt}, from=1-3, to=1-5]
	\arrow[""{name=3, anchor=center, inner sep=0}, "{\alpha_2\cdot\T_2\alpha_1}"', curve={height=24pt}, from=1-3, to=1-5]
	\arrow["{\T_1\T_2\kappa}"'{description},  Rightarrow, shorten >=2pt, shorten <=2pt,from=0, to=1]
	\arrow["\psi"'{description}, shorten >=2pt, shorten <=2pt, Rightarrow, from=2, to=3]
\end{tikzcd} . 
\]
Similarly, equality~\eqref{equation:lax_distr_condition_2_mixed}
follows from
\eqref{equation:distr2}. Equalities~\eqref{equation:psi_unit_1} and
\eqref{equation:psi_unit_2} are simply~\eqref{equation:lax_distr_condition_3_mixed} and
\eqref{equation:lax_distr_condition_4_mixed}, respectively,
specialized to the current situation.
\end{proof}

\begin{proof}[Proof of Proposition~\ref{proposition:pre-abelian}] 
It follows from Corollaries~\ref{corollary:T_2_kernels} 
and~\ref{corollary:T_1_cokernels} that a pointed category equipped
with the colax and lax algebra structures as in the proposition
has cokernels and kernels. Let us prove the converse, assuming that
$\C$ is a pointed category with chosen, normalized cokernels and kernels. 

The components of the 2-cell $\psi$ in~\eqref{Jarka je uz druhy den na
  Sardinii.} are determined by the universal
properties of cokernels and kernels in
diagram~\eqref{diagram:abelian_for_S};  their naturality also
follows from these universal properties. Note that the two commutative
triangles in the upper right corner 
of~\eqref{diagram:abelian_for_S} give precisely
equalities \eqref{equation:distr1} and
\eqref{equation:distr2}. Equalities~\eqref{equation:psi_unit_1} and
\eqref{equation:psi_unit_2} are easy to verify by inspecting the 
diagrams
\[
\xymatrix@R=.6em@C=.6em{
a \ar@{=}[dd]   \ar[rrr]^f&&&b\ar@{=}[dd]  \ar@{->>}[r]  &\coker f
\ar@{=}[rd]^{\psi_{\beta_\C(f)}}
\\
&&&&& \coker f \ar@{=}[d]
\\
a \ar[ddd] \ar[rrr]^f &&&b \ar[ddd] \ar@{->>}[rr]  &&\coker f \ar[ddd]
\\
\\
\\
0 \ar@{=}[rrr]&&& 0 \ar@{=}[rr]&&0
}
\hskip 5em
\xymatrix@R=.6em@C=.6em{
0 \ar@{=}[dd] \ar[rrr]  &&&\ker f\ar@{>->}[dd]  \ar@{=}[r]  &\ker f
\ar@{=}[rd]^{\psi_{\lambda(\beta^\C(f))}}
\\
&&&&& \ker f \ar@{>->}[d]
\\
0 \ar@{=}[ddd] \ar[rrr] &&&a \ar[ddd]^f \ar@{=}[rr]  &&a \ar[ddd]^f
\\
\\
\\
0 \ar[rrr]&&& b \ar@{=}[rr]&&b
}
\]
therefore $\psi$ is a lax rewriting rule by 
Lemma~\ref{lemma:lax_distr_concrete}.

The uniqueness of $\psi$ follows from the exchange rule for the
horizontal composition of 2-cells in the diagram
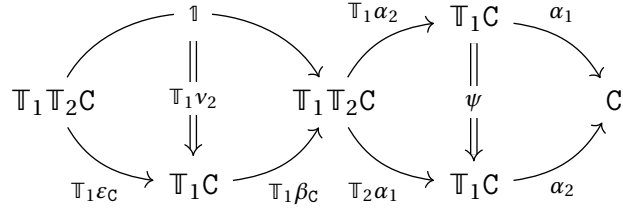
\[\begin{tikzcd}[row sep={3em,between origins},
  column sep={5em,between origins}]
 & & & \T_1\C & \\
	\T_1\T_2\C & & \T_1\T_2\C & &\C \\
     & \T_1\C &&\T_1\C& 
	\arrow["{\id}"{name=0, anchor=center, inner sep=0.5em, description}, 
    from=2-1, to=2-3, in=120, out=60]
	\arrow["{\T_1\epsilon_\C}"'{near end}, from=2-1, to=3-2, in=172, out=-60]
	\arrow["{\T_1\alpha_2}"{near end}, in=-172, out=60, from=2-3, to=1-4]
	\arrow["{\T_1\beta_\C}"'{near start}, in=-120, out=8, from=3-2, to=2-3]
    \arrow["{\T_2\alpha_1}"'{near end}, in=172, out=-60,from=2-3, to=3-4]
	\arrow["{\alpha_2}"'{near start} , in=-120, out=8, from=3-4, to=2-5]
    \arrow["{\alpha_1}" {near start}, in=120, out=-8, from=1-4, to=2-5]
    \arrow["{\T_1\nu_2}"'{description},  shorten >=2pt, shorten <=2pt, Rightarrow, from=0, to=3-2]
	\arrow["\psi"'{description},  Rightarrow, from=1-4, to=3-4]
\end{tikzcd} 
\]
expressed by the commutativity of the diagram
\begin{equation}
\label{Porad prsi a fouka vitr.}
\xymatrix@C=6em@R=3em{\alpha_1 \cdot \T_1\alpha_2  
\ar@{=>}[r]^(.37){\alpha_1 \cdot \T_1\alpha_2 \cdot \T_1 \nu_2 }
\ar@{=>}[d]_\psi
& (\alpha_1 \cdot \T_1\alpha_2 )\cdot (\T_1 \beta_\C 
  \cdot\T_1 \epsilon_\C )   \ar@{=>}[d]^{\psi \cdot \T_1 \beta_\C 
    \cdot\T_1 \epsilon_\C }
\\
\alpha_2 \cdot \T_2\alpha_1  \ar@{=>}[r]^(.37){\alpha_2 \cdot \T_2\alpha_1
  \cdot \T_1\nu_2}
& (\alpha_2 \cdot \T_2\alpha_1) \cdot (\T_1 \beta_\C 
  \cdot\T_1 \epsilon_\C )
}
\end{equation}
of $2$-cells. In that diagram, the vertical right transformation is the
identity by~\eqref{equation:lax_distr_condition_4_mixed}. We verify,
by evaluating on objects of $\Vr\Ar\C$, the equality (with
$\lambda$ suppressed)
\begin{equation}
\label{Zitra se mozna poleti.}
 \T_2\alpha_1\cdot\T_1\nu_2 = 
\nu_2\cdot\T_2\alpha_1
\end{equation}
which implies 
\[
 \alpha_2 \cdot \T_2\alpha_1
  \cdot \T_1\nu_2=\alpha_2\cdot\nu_2\cdot\T_2\alpha_1.
\]
Moreover
\[
\alpha_2\cdot\nu_2\cdot\T_2\alpha_1=\kappa\cdot\T_2\alpha_1
\]
by~\eqref{equation:iota}, therefore the diagram in~\eqref{Porad prsi a fouka
  vitr.}
takes the form
\[
\xymatrix@C=6em@R=3em{\alpha_1 \cdot \T_1\alpha_2  
\ar@{=>}[r]^(.37){\alpha_1 \cdot \T_1\alpha_2 \cdot \T_1 \nu_2 }
\ar@{=>}[d]_\psi
& (\alpha_1 \cdot \T_1\alpha_2 )\cdot (\T_1 \beta_\C 
  \cdot\T_1 \epsilon_\C )   \ar@{=}[d]
\\
\alpha_2 \cdot \T_2\alpha_1  \ar@{=>}[r]^(.37){\kappa \cdot \T_2 \alpha_1}
& (\alpha_2 \cdot \T_2\alpha_1) \cdot (\T_1 \beta_\C 
  \cdot\T_1 \epsilon_\C ).
}
\]
Since the components of the transformation $\kappa$ are monomorphisms,
$\psi$ is determined by the horizontal $2$-cell in~\eqref{Porad prsi a
  fouka  vitr.}. We emphasize that~(\ref{Zitra se mozna poleti.})
holds if $\T_1$ and $\T_2$ are the arrow category monads, not for general
mixed lax rewriting rules.
\end{proof}

\section{Algebraic characterization of abelian categories}
\label{Pres vikend to htozne foukalo.}

In this section we formulate and prove the main application of the theory
we developed. 
Recall~\cite[Section~VII.3]{MacLane} that an additive 
category $\C$ is {\/\em abelian\/} if  
\begin{itemize}[topsep=0pt, partopsep=0pt, itemsep=0pt,  itemindent=-1em]
\item[(i)] it has a null object,
\item [(ii)]
it has all binary biproduct,
\item[(iii)]
every arrow has a kernel and a cokernel, and
\item[(iv)]
every monomorphism is a kernel and every epimorphism is a cokernel.
\end{itemize}

In the following theorem and its proof, the biproduct of objects
$a_1,a_2$ of the category
$\C$ will be denoted by $a_1 \oplus a_2$.
The biproduct 
comes equipped with projections $p_i : a_1 \oplus a_2 \to a_i$ and
coprojections $\iota_i : a_i \to a_1 \oplus a_2$, $i=1,2$.
A morphism $f:a_1 \oplus a_2
\to y$ is determined by the composites 
$f_i := f \iota_i$, $i=1,2$; we write $f
= f_1+f_2$ trusting that the plus signs suggesting a nonexistent
commutativity will not lead to confusion. Likewise, a morphism $g: z \to a_1
\oplus a_2$ is determined by the composites $p_i g$, $i=1,2$,
and we write $g:= (g_1,g_2)$.    

\begin{theorem}
\label{Dnes jsme s Jarkou pekne zmokli.}
An additive category $\C$ is abelian if and only if it is equipped with a
lax $\T_2$-algebra structure $(\C,\eK,\phi_2)$ and a colax
$\T_1$-algebra structure $(\C,\eC,\phi_1)$, such that $\alpha_1$ and
$\alpha_2$ are tied by a~lax rewriting rule $\psi$
such that 
\begin{itemize}[topsep=0pt, partopsep=0pt, itemsep=0pt, leftmargin=2em]
\item[(i)]
its component $\psi_M$ is an epimorphism for each square $M$
\[
\xymatrix@C=2.5em{
a \ar[r]^(.4){(\id,-f)} \ar@{=}[d] & a\oplus b \ar[d]^{0-\id}
\\
a \ar[r]^f & b \ar@{}[lu]|\hole|M
}
\]
with $f$ a monomorphism, and
\item[(ii)]
its component $\psi_N$ is a monomorphism for each square $N$
\[
\xymatrix@C=2.5em{
a \ar[r]^(.4){(-\id,0)} \ar[d]_f & a\oplus b \ar[d]^{-f+\id}
\\
a \ar@{=}[r]& b  \ar@{}[lu]|\hole|N
}
\]
with $f$ an epimorphism.
\end{itemize}
\end{theorem}

\begin{remark} 
It can be shown that conditions (i)--(ii) of Theorem~\ref{Dnes jsme s
  Jarkou pekne zmokli.}  are
equivalent to requiring that the component $\psi_X$ be an isomorphism
for each square $X$
\[
\xymatrix@C=3em{a \ar[r]^u \ar[d]_f  &b \ar[d]^g
\\
c \ar[r]^v&d  \ar@{}[lu]|\hole|X
}
\]
in which $f$ is an epimorphism and $v$ a monomorphism.
\end{remark}

\begin{proof}[Proof of Theorem~\ref{Dnes jsme s Jarkou pekne zmokli.}]
The component $\Psi_M$ of the transformation $\Psi$ featured in
item~(i) of the theorem appears in the diagram
\begin{equation}
\label{inj}
\xymatrix@R=1.2em@C=1.2em{0\ar@{>->}[rrr] \ar@{>->}[dd]  &&&
  \ar@{>->}[dd]^{\kappa_{0-\id}} \Ker(0-\id)  \ar@{=}[r]   
& \Ker(0-\id)   \ar@/^1em/[rd]^{\psi_M}
\\
&&&&& \Ker(\alpha_1(M)) \ar@{>->}[d]^(.5){\kappa_{\alpha_1(M)}}
\\
a \ar[rrr]^{(\id,-f)} \ar@{=}[ddd]  
&&& \ar[ddd]^{0-\id} a\oplus b \ar@{->>}[rr]^(.4){\pi_{(\id,-f)}}  
&& \Coker(\id,-f)\ar@{->>}[ddd]^{\alpha_1(M)}
\\
\\
\\
a \ar@{>->}[rrr]^f &&& b\ar@{->>}[rr]^{\pi_f}  \ar@{}[llluuu]|M   
&& \Coker(f),
}
\end{equation}
the diagram for $\Psi_N$ in~(ii) is similar.  

The crucial observation is that whether $\Psi_M$ is an
isomorphism is independent of the choice of  the kernel of
$0-\id$ and of the cokernel of $(\id,-f)$. This follows from the upper right
square of the `doubling' of~\eqref{inj} portrayed in
Figure~\ref{Greta je na oprave v Dobrisi.}. In that diagram, 
$\Psi'_M$, $\Psi''_M$ and the unnamed
isomorphisms are induced by universal properties, so its upper right
corner commutes.  
\begin{figure}
\[
\xymatrix@C=1.5em@R=1.5em{
&& \Ker(0-1)'' \ar@/^3.8em/[rrrddd]^{\Psi''_M} \ar@{<->}[dd]^\cong 
\ar@{>->}@/_4em/[dddd]_(.2){\kappa''_{0-\id}}|(.48)\hole
\\
0  \ar@{>->}[ddd] \ar@{}[rrrrr]|{\hbox{\tt upper right}}
\ar@{>->}@/^1.5em/[urr]  \ar@{>->}@/_1.5em/[drr] &&&&& 
\\
&&\Ker(0-1)'   \ar@{>->}[dd]^{\kappa'_{0-\id}}  \ar@/^1.5em/[rd]^{\Psi'_M}
\\
&&&  \Ker(\alpha_1'(M))' \ar@{>->}[dd]^(.4){\kappa_{\alpha'_1(M)}}  
\ar @{<->}[rr]^\cong   & & \Ker(\alpha_1''(M))'' 
\ar@{>->}[dd]^(.4){\kappa_{\alpha''_1(M)}}  
\\
a\ar@{>->}[rr]^{(\id,-f)} \ar@{=}[ddd]  && 
a\oplus b \ar@{->>}@/^1.2em/[drrr]^(.65){\pi''_{(\id,-f)}}|(.345)\hole     
\ar@{->>}@/_1.5em/[dr]_{\pi'_{(\id,-f)}}  \ar[ddd]_{0-\id}
\\
&&& \Coker(\id,-f)' \ar@/_2em/[ddr]^{\alpha_1'(M)}  \ar@{<->}[rr]^\cong & & 
\Coker(\id,-f)''  \ar@/^2em/[ddl]_{\alpha_1''(M)}
\\
\\
a\ar@{>->}[rr]^f && b \ar@{->>}[rr]^(.4){\pi_f}  && \Coker(f)
}
\]
\caption{The doubling of~\eqref{inj}.\label{Greta je na oprave v Dobrisi.}}
\end{figure}
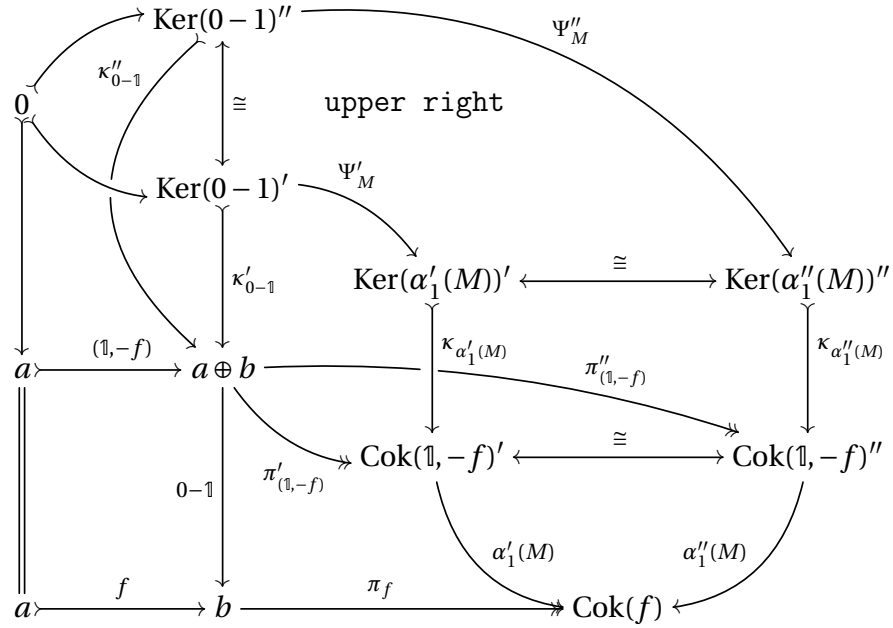
We may thus assume that the kernels and cokernels are chosen as in
\begin{equation}
\label{Pozitri letim do Sevily.}
\xymatrix@C=5em@R=2em{
0 \ar@{>->}[r]\ar@{>->}[dd]&
a \ar[dd]_{(\id,0)}\ar[ddr]|f|\hole  \ar@/^1.2em/[dr]^{\Psi_M}
\\
&& \Ker(-\pi_f) \ar@{>->}[d]^{\kappa_{-\pi_f}}
\\
a \ar@{=}[dd]  \ar[r]^(.4){(\id,-f)}    & a\oplus b 
\ar[dd]^{0-\id} \ar[r]^{f+1}   &b \ar@{->>}[dd]^{-\pi_f =\alpha_1(M)}
\\
\\
a \ar@{>->}[r] & b \ar@{->>}[r]^(.4){\pi_f} & \Coker(f).
}
\end{equation}

Let us come to the actual proof. 
Assume that $\C$ is equipped with lax-colax algebra structures as in the
theorem. By Proposition~\ref{proposition:pre-abelian}, it then has
kernels and cokernels. We need to verify that item (i) implies that the
monomorphism  $f$ is a kernel. Since $f$ is assumed to be a
monomorphism, the upper right corner of~(\ref{Pozitri letim do
  Sevily.}) forces $\Psi_M$ to be a monomorphism as well, and hence,
by assumption, it is
actually an isomorphism. It follows that $f$ is a kernel of
$-\pi_f$. The proof that~(ii) implies that each epimorphism is a
cokernel is similar.

Let us prove the converse.
Each abelian category has lax and colax algebra
structures as in the theorem by Proposition~\ref{proposition:pre-abelian}.
Let us prove that~(i) is fulfilled for each monomorphism $f$. The
yoga of abelian categories teaches us that such an $f$ is a kernel of
its cokernel $\pi_f$, so it is also a kernel of $-\pi_f$. Thus
$\Psi_M$ in~(\ref{Pozitri letim do Sevily.}) must be an
isomorphism. Item~(ii) is handled similarly.   
\end{proof}

\appendix

\section{(Co)lax algebras and morphisms}
\label{Druhy den v Brne.}
 
We recall standard definitions of lax and colax algebras, their
morphisms, and related categories. Although in some places we repeat
definitions that have already been given in the main text, we believe
it is convenient to have all of them in one place. 
We restrict ourselves to the normalized
case, as this suffices for our purposes. Some of the material is taken
from~\cite{monads}.

\begin{definition}[Lax algebra]
\label{definition:lax_alg}
Let $(\T,\mu,\eta)$ be a $2$-monad on a $2$-category $\K$,
where \hbox{$\mu : \T^2 \Rightarrow \T$} and $\eta : \id_\K
\Rightarrow \T$ denote the multiplication and unit, respectively. 
A (normalized) lax $\T$-algebra structure on $A \in \K$ is an arrow 
$\alpha\colon \T A\to A$ together with a 2-cell
    \[\phi\colon \alpha \cdot\T \alpha
\Longrightarrow \alpha\cdot \mu_A\] 
as in
\[
  \xymatrix@R=2.5em@C=2.5em{
    \T^2 A \ar[r]^{\T \alpha} \ar[d]_{\mu_A} 
    & \T A \ar[d]^{\alpha} \ar@{=>}[dl]|{\phi\rule{.1em}{0pt}}
    \\
    \T A \ar[r]^{\alpha} & A
  }
\]
satisfying the following four conditions:

\begin{itemize}[topsep=0pt, partopsep=0pt, itemsep=0pt, itemindent=-1em]
\item[(i)]
$\alpha \cdot \eta_A=\uu_A$,
\item[(ii)] 
there is an equality of 2-cells
\[
\xymatrix@R=1.2em@C=1.7em{
\T^3 A \ar[rr]^{\T^2 \alpha}\ar[dd]_{\mu_{\T A}} && \T^2 A \ar[rd]^{\T \alpha} \ar[dd]_{\mu_A}
\\
&&& \T A \ar[dd]^{\alpha} \ar@{=>}[ld]|{\phi\rule{.1em}{0pt}}
\\
\T^2 A  \ar[rr]^{\T \alpha} \ar[rd]_{\mu_A} && \T A
\ar@{=>}[ld]|{\phi\rule{.1em}{0pt}}
\ar[dr]^{\alpha}
\\
&\T A \ar[rr]^{\alpha} && A
}
\hskip 1em
\raisebox{-3em}{$=$} \hskip 1em
\xymatrix@R=1.2em@C=1.7em{
\T^3 A \ar[rr]^{\T^2 \alpha}\ar[dd]_{\mu_{\T A}} \ar[rd]^(.6){\T\mu_A}   && 
\T^2 A \ar[rd]^{\T \alpha} \ar@{=>}[ld]|(.45){\T{\phi\rule{.1em}{0pt}}}
\\
&\T^2 A  \ar[rr]^{\T \alpha}   \ar[dd]_{\mu_A}&& \T A \ar[dd]^{\alpha} 
\ar@{=>}[lldd]|{\phi\rule{.1em}{0pt}}
\\
\T^2 A \ar[rd]_{\mu_A} &&
\\
&\T A \ar[rr]^{\alpha} && A,
}
\]
\item[(iii)]
$\phi \cdot \T\eta_A=\uu_{\alpha}$, and
\item[(iv)]
$\phi \cdot \eta_{\T A}=\uu_{\alpha}$.
\end{itemize}
\end{definition}

\begin{definition}[Colax algebra]
\label{definition:colax_alg}
%Let $(\T,\mu,\eta)$ be a $2$-monad over a $2$-category a category $\C$,
A (normalized) colax $\T$-algebra structure on $A\in \K$ is
an arrow
$\alpha\colon \T A\to A$ together with a 2-cell
\[\phi\colon \alpha\cdot \mu_A \Longrightarrow \alpha \cdot\T \alpha\]
as in
\[
\xymatrix@R=2.5em@C=2.5em{
\T^2 A \ar[r]^{\T \alpha} \ar[d]_{\mu_A} 
& \T A \ar[d]^{\alpha} \ar@{<=}[dl]|{\phi\rule{.1em}{0pt}}
\\
\T A \ar[r]^{\alpha} & A
}
\]
satisfying the following four conditions:
\begin{itemize}[topsep=0pt, partopsep=0pt, itemsep=0pt, itemindent=-1em]
\item[(i)]
$\alpha \cdot \eta_A=\uu_A$,
\item[(ii)] 
there is an equality of 2-cells
\[
\xymatrix@R=1.2em@C=1.7em{
\T^3 A \ar[rr]^{\T^2 \alpha}\ar[dd]_{\mu_{\T A}} && \T^2 A \ar[rd]^{\T \alpha} \ar[dd]_{\mu_A}
\\
&&& \T A \ar[dd]^{\alpha} \ar@{<=}[ld]|{\phi\rule{.1em}{0pt}}
\\
\T^2 A  \ar[rr]^{\T \alpha} \ar[rd]_{\mu_A} && \T A
\ar@{<=}[ld]|{\phi\rule{.1em}{0pt}}
\ar[dr]^{\alpha}
\\
&\T A \ar[rr]^{\alpha} && A
}
\hskip 1em
\raisebox{-3em}{$=$} \hskip 1em
\xymatrix@R=1.2em@C=1.7em{
\T^3 A \ar[rr]^{\T^2 {\alpha}}\ar[dd]_{\mu_{\T A}} \ar[rd]^(.6){\T\mu_A}   && 
\T^2 A \ar[rd]^{\T {\alpha}} \ar@{<=}[ld]|(.5){\T{\phi\rule{.1em}{0pt}}}
\\
&\T^2 A  \ar[rr]^{\T {\alpha}}   \ar[dd]_{\mu_A}&& \T A \ar[dd]^{\alpha} 
\ar@{<=}[lldd]|{\phi\rule{.1em}{0pt}}
\\
\T^2 A \ar[rd]_{\mu_A} &&
\\
&\T A \ar[rr]^{\alpha} && A,
}
\]
\item[(iii)]
$\phi \cdot \T\eta_A=\uu_{\alpha}$, and
\item[(iv)]
$\phi \cdot \eta_{\T A}=\uu_{\alpha}$.
\end{itemize}
\end{definition}

\begin{definition}[Lax morphism of lax algebras]
\label{definition:lax_alg_lax_morph}
Let $(A,\alpha_1,\phi_1)$ and $(B,\alpha_2,\phi_2)$ be lax
$\T$-algebras as in Definition~\ref{definition:lax_alg}. 
A lax morphism
\[
(f,\theta)\colon (A,\alpha_1,\phi_1)\to (B,\alpha_2,\phi_2)
\]
is a pair consisting of an arrow
$f\colon A\to B$ and a 2-cell
\[
\theta\colon \alpha_2\cdot \T f  \Rightarrow f\cdot \alpha_1
\]
as in
\[
\xymatrix@R=2.5em@C=2.5em{
\T A \ar[r]^{\T f} \ar[d]_{\alpha_1} 
& \T B\ar[d]^{\alpha_2} \ar@{=>}[dl]|{\theta\rule{.1em}{0pt}}
\\
A \ar[r]_{f} & B
}
\]
satisfying the following two conditions:
\begin{itemize}[topsep=0pt, partopsep=0pt, itemsep=0pt, itemindent=-1em]
\item[(i)] 
there is an equality of 2-cells
\[
\xymatrix@R=1.2em@C=2em{
\T^2 A \ar[rr]^{\T^2 f}\ar[dd]^{\mu_A}
&&\T^2  B \ar[dd]^{\mu_B} \ar[rd]^{\T\alpha_2}
\\
&&& \T  A  \ar[dd]^{\alpha_2}
\ar@{=>}[ld]|{\phi_2\rule{.1em}{0pt}}
\\
\T  A \ar[rr]^{\T f} \ar[rd]_{\alpha_1} && \T B
\ar@{=>}[ld]|{\theta\rule{.1em}{0pt}}
\ar[dr]^{\alpha_2}
\\
&A \ar[rr]^{f} && B
}
\hskip .5em
\raisebox{-3em}{$=$} \hskip .5em
\xymatrix@R=1.2em@C=2em{
\T^2  A \ar[rr]^{\T^2 f}\ar[dd]^{\mu_A} 
\ar[rd]^(.6){\T \alpha_1}   && 
\T B  \ar[rd]^{\T \alpha_2} \ar@{=>}[ld]|(.45){\T{\theta\rule{.1em}{0pt}}}
\\
&\T A  \ar[rr]^{\T f}   \ar[dd]_{\alpha_1}
\ar@{=>}[ld]|{\T\phi_1\rule{.1em}{0pt}}
&& \T B 
\ar[dd]^{\alpha_2} 
\ar@{=>}[lldd]|{\theta\rule{.1em}{0pt}}
\\
\T A \ar[rd]_{\alpha_1} &&
\\
& A \ar[rr]^{f} && B,
}
\]
\item[(ii)] 
$\theta\cdot\eta_A=\id_f$.
\end{itemize}
\end{definition}

\begin{definition}[Colax morphism of lax algebras]
\label{Zitra je Den otevreneho letiste.}
Let $(A,\alpha_1,\phi_1)$ and $(B,\alpha_2,\phi_2)$ be lax
$\T$-algebras. %A~(normalized)
A colax morphism
\[
(f,\theta)\colon (A,\alpha_1,\phi_1)\to (B,\alpha_2,\phi_2)
\] 
is a pair consisting of an arrow $f\colon A\to B$ and a 2-cell
\[
\theta\colon f\cdot \alpha_1 \Rightarrow \alpha_2\cdot \T f
\]
as in
\[
\xymatrix@R=2.5em@C=2.5em{
\T A \ar[r]^{\T f} \ar[d]_{\alpha_1} 
& \T B\ar[d]^{\alpha_2} \ar@{<=}[dl]|{\theta\rule{.1em}{0pt}}
\\
A \ar[r]_{f} & B
}
\]
satisfying the following two conditions:
\begin{itemize}[topsep=0pt, partopsep=0pt, itemsep=0pt, itemindent=-1em]
\item[(i)]
there is an equality of 2-cells
\[
\xymatrix@R=1.2em@C=1.7em{
\T^2 A \ar[rr]^{\T^2 f}\ar[dd]_{\T \alpha_1} && \T^2 B \ar@{<=}[lldd]|{\T\theta\rule{.1em}{0pt}}\ar[rd]^{\mu_B} \ar[dd]_{\T \alpha_2}
\\
&&& \T B \ar[dd]^{\alpha_2} \ar@{<=}[ld]|{\phi_2\rule{.1em}{0pt}}
\\
\T A  \ar[rr]^{\T f} \ar[rd]_{\alpha_1} && \T B
\ar@{<=}[ld]|{\theta\rule{.1em}{0pt}}
\ar[dr]^{\alpha_2}
\\
&A \ar[rr]^f && B
}
\hskip 1em
\raisebox{-3em}{$=$} \hskip 1em
\xymatrix@R=1.2em@C=1.7em{
\T^2 A \ar[rr]^{\T^2 f}\ar[dd]_{{\T \alpha_1}} \ar[rd]^(.6){\mu_A}   && 
\T^2 B \ar[rd]^{\mu_B} 
\\
&\T A \ar@{<=}[ld]|{\phi_1\rule{.1em}{0pt}} \ar[rr]^{\T f}   \ar[dd]_{\alpha_1}&& \T B \ar[dd]^{\alpha_2} 
\ar@{<=}[lldd]|{\theta\rule{.1em}{0pt}}
\\
\T A \ar[rd]_{\alpha_1} &&
\\
&A \ar[rr]^f && B,
}
\]
\item[(ii)] 
$\theta\cdot\eta_A=\id_f$.
\end{itemize} 
\end{definition}

\begin{definition}[Lax morphism of colax algebras]
\label{definition:lax_morph_of_colax_alg}
Let $(A,\alpha_1,\phi_1)$ and $(B,\alpha_2,\phi_2)$ be 
colax $\T$-al\-gebras as in Definition~\ref{definition:colax_alg}. 
%A~(normalized) 
A lax morphism 
\[
(f,\theta)\colon (A,\alpha_1,\phi_1)\to (B,\alpha_2,\phi_2)
\]
a pair consisting of an arrow 
$f\colon A\to B$ and a 2-cell 
\[
\theta\colon \alpha_2\cdot \T f  \Rightarrow f\cdot \alpha_1
\]
as in
\begin{equation}
\xymatrix@R=2.5em@C=2.5em{
\T A \ar[r]^{\T f} \ar[d]_{\alpha_1} 
& \T B\ar[d]^{\alpha_2} \ar@{=>}[dl]|{\theta\rule{.1em}{0pt}}
\\
A \ar[r]_{f} & B
}
\end{equation}satisfying the two following conditions:
\begin{itemize}[topsep=0pt, partopsep=0pt, itemsep=0pt, itemindent=-1em]
\item[(i)]
there is an equality of 2-cells
\[
\xymatrix@R=1.2em@C=1.7em{
\T^2 A \ar[rr]^{\T^2 f}\ar[dd]_{\T \alpha_1} && \T^2 B \ar@{=>}[lldd]|{\T\theta\rule{.1em}{0pt}}\ar[rd]^{\mu_B} \ar[dd]_{\T \alpha_2}
\\
&&& \T B \ar[dd]^{\alpha_2} \ar@{=>}[ld]|{\phi_2\rule{.1em}{0pt}}
\\
\T A  \ar[rr]^{\T f} \ar[rd]_{\alpha_1} && \T B
\ar@{=>}[ld]|{\theta\rule{.1em}{0pt}}
\ar[dr]^{\alpha_2}
\\
&A \ar[rr]^f && B
}
\hskip 1em
\raisebox{-3em}{$=$} \hskip 1em
\xymatrix@R=1.2em@C=1.7em{
\T^2 A \ar[rr]^{\T^2 f}\ar[dd]_{{\T \alpha_1}} \ar[rd]^(.6){\mu_A}   && 
\T^2 B \ar[rd]^{\mu_B} 
\\
&\T A \ar@{=>}[ld]|{\phi_1\rule{.1em}{0pt}} \ar[rr]^{\T f}   \ar[dd]_{\alpha_1}&& \T B \ar[dd]^{\alpha_2} 
\ar@{=>}[lldd]|{\theta\rule{.1em}{0pt}}
\\
\T A \ar[rd]_{\alpha_1} &&
\\
&A \ar[rr]^f && B,
}
\]
\item[(ii)] $\theta\cdot\eta_A=\id_f$.
\end{itemize} 
\end{definition}

\begin{definition}\label{definition:2-cell}
A 2-cell between two lax $\T$-algebra morphisms $(f,\theta)\Rightarrow
(g,\delta)$ is a 2-cell $\rho \colon f\Rightarrow g$, such that the
following two composites of 2-cells equal:
\begin{equation}
    \label{equation:2-cell}
\xymatrix@R=.2em{
&&
\\
\T A  \ar[ddddddd]_{\alpha_1} 
\ar@/^1em/[rr]^{\T g}  \ar@/_1.2em/[rr]_{\T f} && 
\T B  \ar[ddddddd]^{\alpha_2} 
\\
& \ar@{=>}[uu]_{\T\rho}
\\
\\
\\
\\
\\
\\
A \ar[rr]^f \ar@{=>}[rruuuuuuu]|\hole|\theta  && B
}
\raisebox{-4.5em}{\ $=$}
\raisebox{-1em}{
\xymatrix@R=.2em{
\T A  \ar[ddddddd]_{\alpha_1}  
\ar[rr]^{\T g}  &&  
\T B  \ar[ddddddd]^{\alpha_2} 
\\
\\
\\
\\
\\
\\ &&
\\
A  \ar@{=>}[rruuuuuuu]|\hole|\delta
\ar@/^1em/[rr]^{g}  \ar@/_1.2em/[rr]_{f} && 
B  .
\\
& \ar@{=>}[uu]_{\rho}
}
}
\end{equation}
\end{definition}

\begin{definition}\label{definition:Lax-T-alg_lax}
    The 2-category $\tt{Lax}\hbox{-}\T\tt{\hbox{-}alg}_{\tt{lax}}$ 
has
\begin{itemize}[topsep=0pt, partopsep=0pt, itemsep=0pt, itemindent=-2em]
        \item lax $\T$-algebras of Definition~\ref{definition:lax_alg}
          as its objects,
        \item lax algebra morphisms of
          Definition~\ref{definition:lax_alg_lax_morph} as its arrows, and
        \item 
    $2$-cells of Definition~\ref{definition:2-cell} as its $2$-cells. 
    \end{itemize}
\end{definition}

\begin{definition}
    The 2-category $\tt{Colax}\hbox{-}\T\tt{\hbox{-}alg}_{\tt{lax}}$ 
has
    \begin{itemize}[topsep=0pt, partopsep=0pt, itemsep=0pt, itemindent=-2em]
        \item colax $\T$-algebras of Definition~\ref{definition:colax_alg}
          as its objects,
        \item lax algebra morphisms of
          Definition~\ref{definition:lax_morph_of_colax_alg} as its arrows, and
        \item 
    $2$-cells of Definition~\ref{definition:2-cell} as its $2$-cells. 
    \end{itemize}
\end{definition}

\section{Proof of Theorem~\ref{Ozivil jsem Jarcino auto.}}
\label{Ta nova strecha muj program tohoto tydne velmi zahustuje.}

\def\OOO{A_{000}}      \def\OAO{f_{0?0}}  \def\BB{\phi_B}
\def\OIO{A_{010}}      \def\OIA{f_{01?}}  \def\EE{\phi_E}
\def\OII{A_{011}}      \def\AII{f_{?11}}  \def\SS{\phi_S}
\def\IIO{A_{110}}      \def\AOO{f_{?00}}  \def\NN{\phi_N}
\def\III{A_{111}}      \def\IAI{f_{1?1}}  \def\WW{\phi_W}
\def\IOO{A_{100}}      \def\IOA{f_{10?}}  \def\FF{\phi_F}
\def\IOI{A_{101}}      \def\IIA{f_{11?}}
\def\OOI{A_{001}}      \def\AIO{f_{?10}}
\def\IAO{f_{1?0}}
\def\OOA{f_{00?}}
\def\AOI{f_{?01}} 
\def\OAI{f_{0?1}} 
This part of the Appendix is devoted to the proof of 
Theorem~\ref{Ozivil jsem Jarcino auto.}. 
We begin with the most difficult part, namely the verification
of equality (ii) of Definition~\ref{definition:lax_alg}.
It has the form
\begin{subequations}
\begin{equation}
\label{Mam rozbite kolecko.}
  \xymatrix@R=1.2em@C=2em{
    \OOO \ar[rr]^{\OAO}\ar[dd]_{\AOO}
    &&\OIO \ar[dd]^{\AIO} \ar[rd]^{\OIA}
    \ar@{=>}[lldd]|{\BB}
    \\
    &&& \OII  \ar[dd]^{\AII}
    \ar@{=>}[ld]|{\EE\rule{.1em}{0pt}}
    \\
    \IOO\ar[rr]^{\IAO} \ar[rd]_{\IOA} && \IIO
    \ar@{=>}[ld]|{\SS\rule{.1em}{0pt}}
    \ar[dr]^{\IIA}
    \\
    &\IOI \ar[rr]^{\IAI} && \III
  }
  \hskip .5em
  \raisebox{-3em}{$=$} \hskip .5em
  \xymatrix@R=1.2em@C=2em{
    \OOO \ar[rr]^{\OAO}\ar[dd]_{\AOO} 
    \ar[rd]^(.6){\OOA}   && 
    \OIO \ar[rd]^{\OIA} \ar@{=>}[ld]|\NN
    \\
    &\OOI \ar[rr]^{\OAI}   \ar[dd]_{\AOI}
    \ar@{=>}[ld]|{\WW\rule{.1em}{0pt}}
    && \OII 
    \ar[dd]^{\AII} 
    \ar@{=>}[lldd]|{\FF\rule{.1em}{0pt}}
    \\
    \IOO\ar[rd]_{\IOA} &&
    \\
    &\IOI \ar[rr]^{\IAI} && \III.
  }
\end{equation}
which is a diagrammatic expression of the equality
\begin{equation}
\label{Uvidim jestli se Vojta ozve.}
\SS \cdot \AOO \ \Box \ \IIA \cdot \BB \ \Box \ \EE  \cdot  \OAO   
=
\IAI \cdot \WW \ \Box \ \FF \cdot \OOA \ \Box \ \AII \cdot \NN
\end{equation}
of the vertical composites of $2$-cells. The arrows
featured above decorate the vertices and oriented edges of the
$3$-dimensional cube
\begin{equation}
\label{Pokud ne, zavolam mu ve ctvrtek.}
\xymatrix@R=1.3em@C=1.3em{
\OOO \ar[rrr]^(.5)\OAO \ar[ddd]_{\AOO} \ar[rd]^(.6)\OOA &&& 
\OIO  \ar[ddd]_(.6){\AIO}|(.37)\hole  \ar[rd]^\OIA
\\
& \OOI \ar[rrr]^(.4)\OAI  \ar[ddd]_(.4)\AOI   &&& \OII  \ar[ddd]_(.4)\AII
\\
\\
\IOO  \ar[rrr]^(.6)\IAI|(.39)\hole \ar[rd]^(.6)\IOA &&& \IIO \ar[rd]^\IIA 
\\
& \IOI \ar[rrr]^\IAI &&& \III
}
\end{equation}
\end{subequations}
whose faces are decorated by the $2$-cells as in~(\ref{Mam rozbite
  kolecko.}).  We say that the decorated cube in~(\ref{Pokud ne,
  zavolam mu ve ctvrtek.}) {\/\em commutes\/} if~(\ref{Uvidim
  jestli se Vojta ozve.}) holds. We call the cube decorated
as in~(\ref{Kdy bude zas letadelko v poradku?}) the {\/\em coherence
  cube\/} for $\alpha_i$, $i=1,2$.

The proof of (ii) will rely on the obvious
fact that a big cube composed of smaller commutative cubes with matching
decorations of the adjacent faces commutes too. In our case we consider
the `big cube' portrayed in Figure~\ref{Jak bude v sobotu?},
\begin{figure}
\centering
\xymatrix@C=1em{
{{\ttt33}A} \ar[rrr]^{{\ttt32} \alpha_1} \ar[ddd]_{\TT23 \mu_1 \T_1}
\ar[rd]^(.6){{\ttt31} \mu_1}
&&& {{\ttt32}A}  \ar[rrr]^{{\ttt22} \alpha_2} 
\ar@{--}[dddddd]|(.33){\rule{0em}{2.5em}}|(.82){\rule{0em}{1.5em}}
%|(.2){A\rule{0em}{1.5em}}   
\ar@{--}[rrdd]   &&&  
{{\ttt22}A} \ar[ddd]^{\TT22 \mu_1}|(.67)\hole \ar[rd]^{{\ttt21} \alpha_1}
\\
&\ar@{--}[rrrrrr] \ar@{--}[dddddd] 
{{\ttt32} A} \ar[rd]^{{\ttt12} \mu_2}    &&&%\ar@{--}[dddddd] %NOVE PRIDANA
&&& 
{{\ttt21} A} \ar[rd]^{{\ttt11} \alpha_2} \ar@{--}[dddddd] 
\\
&& 
{{\ttt22} A}  \ar[rrr]^{{\ttt21} \alpha_1}  \ar[ddd]^{{\ttt20} \mu_1}
&&&
\ar@{--}[dddddd]
{{\ttt21}A} \ar[rrr]^{{\ttt11} \alpha_2}   &&& {{\ttt11} A}
\ar[ddd]^{{\ttt10} \alpha_1}
\\
\ar@{--}[rrrrrr]|(.37){\rule{2em}{0em}}|(.82){\rule{2em}{0em}} \ar@{--}[rrdd]
{{\ttt32} A}  \ar[ddd]_{{\ttt02} \mu_2 {\ttt10}} &&&&&& {{\ttt21} A}
\ar[ddd]^{{\ttt01}\mu_2}
\ar@{--}[rrdd]
\\
& %\ar@{--}[rrrrrr]&&&&&& %NOVE PRIDANA TECKOVANA
\\
&&{{\ttt21} A} \ar[ddd]^{{\ttt01}\mu_2} \ar@{--}[rrrrrr] 
&&&&&& {{\ttt10}A} \ar[ddd]^{\alpha_2}
\\
{{\ttt22}A}   \ar[rrr]^{{\ttt21} \alpha_1}|(.67)\hole \ar[rd]_{{\ttt20} \mu_1}   &&& \ar@{--}[rrdd]
{{\ttt21} A} \ar[rrr]^(.35){{\ttt11} \alpha_2} &&& {{\ttt11} A} \ar[rd]^{{\ttt10} \alpha_1}
\\
& \ar@{--}[rrrrrr]
{{\ttt21} A} \ar[rd]_{{\ttt01}\mu_2} &&&&&& {{\ttt10} A} \ar[rd]^{\alpha_2}
\\
&& {{\ttt11}A} \ar[rrr]^{{\ttt10} \alpha_1}  &&&{{\ttt10} 
A} \ar[rrr]^{\alpha_2}&&& {A} 
}
\caption{The big cube.\label{Jak bude v sobotu?}}
\end{figure}
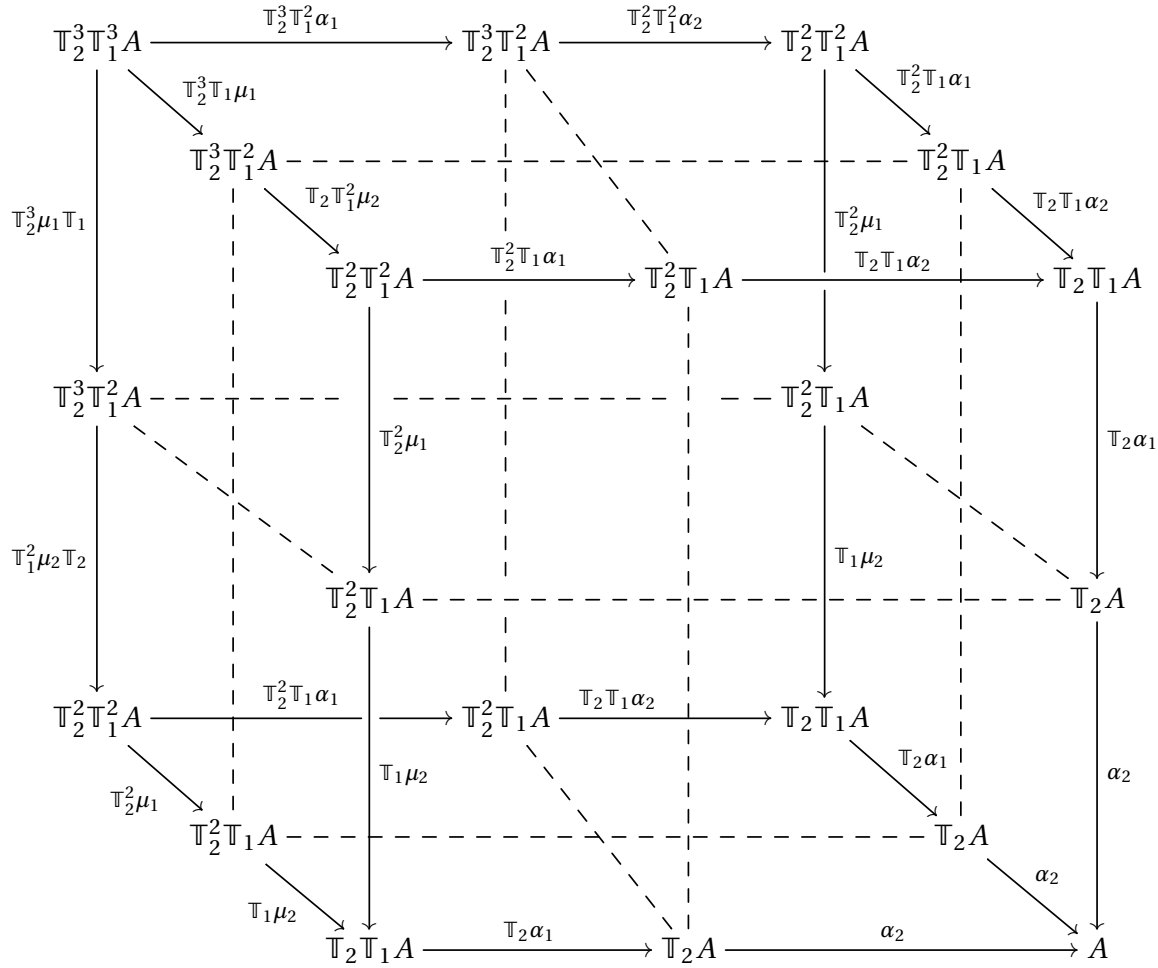
which is the coherence cube for the composed
$\T_{21}$-algebra~(\ref{Dnes divadlo vynecha.}) with the associativity
constraint $\phi_{21}$ introduced in \eqref{Martin prijel do Brna}. In
the remainder of this part of the proof
we list and comment on the eight cubes that constitute it an explain why they
commute. 

Since we assume that the distributivity law is invertible, we will
simplify the notation by making no distinction between the composites 
$\T_1^m\T_2^n$ and $\T_2^n\T_1^m$, $m,n \geq 0$, and moving the
iterated~$\T_2$'s to
the left wherever it makes sense. We will also assume the
simplified identities~\eqref{Jdu si koupit podlozky.}.

\vskip .5em
\noindent 
A. The upper left back cube is the image of the coherence
cube~\eqref{Kdy bude zas letadelko v poradku?} 
for $\alpha_1$ under~$\TT23$, thus it commutes.

\vskip .5em
\noindent 
B. The upper right back cube is the image
of the cube represented in~\eqref{B}  under $\ttt20$, thus it commutes
by assumption.

\vskip .5em
\noindent 
C. The upper front left cube equals $\ttt10$ applied to the cube
represented by  
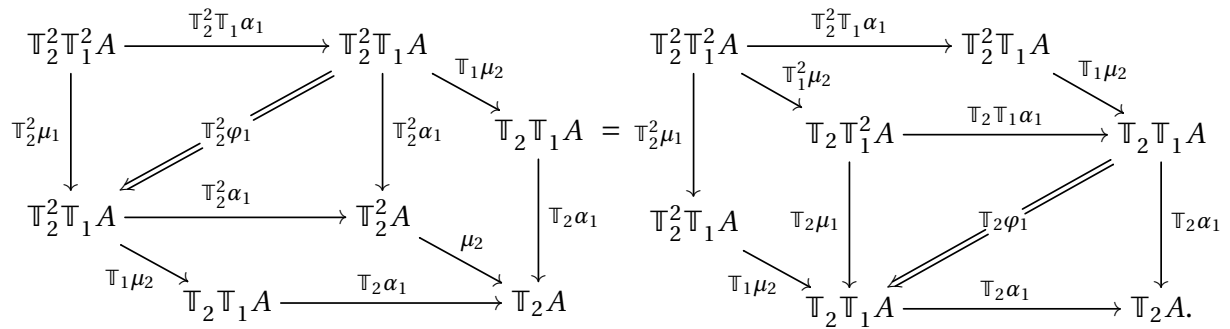
\def\OOO{{\ttt22}A}      \def\OAO{\TT22\TT1{}\alpha_1}  
\def\BB{{\ttt20}\phi{_1}}
\def\OIO{\TT22\TT1{}A}      \def\OIA{{\ttt01}\mu_2}  \def\EE{\phi_E}
\def\OII{\TT2{}\TT1{}A}      \def\AII{{\ttt10}\alpha_1}  \def\SS{\phi_S}
\def\IIO{{\ttt20}A}           \def\AOO{{\ttt20}\mu{_1}}  \def\NN{}
\def\III{{\ttt10}A}         \def\IAI{{\ttt10}\alpha_1}  \def\WW{\phi_W}
\def\IOO{\TT22\TT1{}A}      \def\IOA{{\ttt01}\mu{_2}}  \def\FF{{\ttt10}\phi{_1}}
\def\IOI{\TT2{}\TT1{}A}      \def\IIA{\mu{_2}}
\def\OOI{{\ttt12}A}      \def\AIO{{\ttt20}\alpha{_1}}
\def\IAO{{\ttt20}\alpha{_1}}
\def\OOA{{\ttt02}\mu{_2}}
\def\AOI{{\ttt10}\mu{_1}} 
\def\OAI{{\ttt11} \alpha{_1}} 
\[
  \xymatrix@R=1.2em@C=1.7em{
    \OOO \ar[rr]^{\OAO}\ar[dd]_{\AOO}
    &&\OIO \ar[dd]^{\AIO} \ar[rd]^{\OIA}
    \ar@{=>}[lldd]|{\BB}
    \\
    &&& \OII  \ar[dd]^{\AII}
    %\ar@{=>}[ld]|{\EE\rule{.1em}{0pt}}
    \\
    \IOO\ar[rr]^{\IAO} \ar[rd]_{\IOA} && \IIO
    %\ar@{=>}[ld]|{\SS\rule{.1em}{0pt}}
    \ar[dr]^{\IIA}
    \\
    &\IOI \ar[rr]^{\IAI} && \III
  }
  \hskip-.3em\raisebox{-3em}{$=$} \hskip .1em
  \xymatrix@R=1.2em@C=1.7em{
    \OOO \ar[rr]^{\OAO}\ar[dd]_{\AOO} 
    \ar[rd]^(.6){\OOA}   && 
    \OIO \ar[rd]^{\OIA} %\ar@{=>}[ld]|\NN
    \\
    &\OOI \ar[rr]^{\OAI}   \ar[dd]_{\AOI}
    %\ar@{=>}[ld]|{\WW\rule{.1em}{0pt}}
    && \OII 
    \ar[dd]^{\AII} 
    \ar@{=>}[lldd]|{\FF\rule{.1em}{0pt}}
    \\
    \IOO\ar[rd]_{\IOA} &&
    \\
    &\IOI \ar[rr]^{\IAI} && \III.
  }
\]
Its commutativity follows from the $2$-naturality of $\mu_2$.

\vskip .5em
\noindent 
D. The upper front right cube is the image of the cube represented
in~\eqref{D} under $\ttt10$, hence it commutes by assumption. 

\vskip .5em
\noindent 
E. The commutativity of the bottom back left cube
means the equality
%VRCHOLY
\def\OOO{{\ttt32}A}      
\def\OIO{{\ttt31}A}     
\def\OII{{\ttt30}A}      
\def\III{{\ttt20}A}        
\def\IOO{\TT22\TT12A}      
\def\IOI{{\ttt21}A}    
\def\IIO{{\ttt21}A}         
\def\OOI{{\ttt31}A}     
%HRANY
\def\OAO{{\ttt31}\alpha_1}  
\def\OIA{{\ttt30}\alpha_1}  
\def\AII{\mu{_2}{\ttt10}} 
\def\AOO{{\ttt02}\mu{_2}{\ttt10}}  
\def\IOA{{\ttt20}\mu{_1}} 
\def\IAI{{\ttt20}\alpha_1}  
\def\IAO{{\ttt21}\alpha{_1}}
\def\IIA{{\ttt20}\alpha{_1}}
\def\AIO{{\ttt01}\mu{_2}{\ttt10}}
\def\OAI{{\ttt30} \alpha{_1}} 
\def\AOI{{\ttt01}\mu{_2}{\ttt10}} 
\def\OOA{{\ttt30}\mu{_1}}
%STENY
\def\BB{{\ttt20}(\phi)}
\def\NN{{\ttt30}\phi{_1}\hskip-.4em}
\def\WW{\phi_W}
\def\FF{\TT1{}(\psi)}
\def\SS{{\ttt20}\phi_1 \hskip-.4em}
\def\EE{\psi}
\[
  \xymatrix@R=1.2em@C=1.7em{
    \OOO \ar[rr]^{\OAO}\ar[dd]_{\AOO}
    &&\OIO \ar[dd]^{\AIO} \ar[rd]^{\OIA}
    %\ar@{=>}[lldd]|{\BB}
    \\
    &&& \OII  \ar[dd]^{\AII}
   % \ar@{=>}[ld]|{\EE\rule{.1em}{0pt}}
    \\
    \IOO\ar[rr]^{\IAO} \ar[rd]_{\IOA} && \IIO
    \ar@{=>}[ld]|{\SS\rule{.1em}{0pt}}
    \ar[dr]^{\IIA}
    \\
    &\IOI \ar[rr]^{\IAI} && \III
  }
  \hskip-.7em 
  \raisebox{-3em}{$=$} \hskip .1em
  \xymatrix@R=1.2em@C=1.7em{
    \OOO \ar[rr]^{\OAO}\ar[dd]_{\AOO} 
    \ar[rd]^(.6){\OOA}   && 
    \OIO \ar[rd]^{\OIA} 
\ar@{=>}[ld]|\NN
    \\
    &\OOI \ar[rr]^{\OAI}   \ar[dd]_{\AOI}
    %\ar@{=>}[ld]|{\WW\rule{.1em}{0pt}}
    && \OII 
    \ar[dd]^{\AII} 
  %  \ar@{=>}[lldd]|{\FF\rule{.1em}{0pt}}
    \\
    \IOO\ar[rd]_{\IOA} &&
    \\
    &\IOI \ar[rr]^{\IAI} && \III
  }
\]
which follows from the $2$-naturality of $\mu_2$. 

\vskip .5em
\noindent 
F. The commutativity of the bottom back right cube
means the equality
%VRCHOLY
\def\OOO{{\ttt31}A}      
\def\OIO{{\ttt21}A}     
\def\OII{{\ttt20}A}      
\def\III{{\ttt10}A}        
\def\IOO{{\ttt21}A}      
\def\IOI{{\ttt20}A}    
\def\IIO{{\ttt11}A}         
\def\OOI{{\ttt30}A}     
%HRANY
\def\OAO{{\ttt21}\alpha_2}  
\def\OIA{{\ttt20}\alpha_1}  
\def\AII{\mu{_2}} 
\def\AOO{{\ttt01}\mu{_2}{\ttt10}}  
\def\IOA{{\ttt20}\alpha{_1}} 
\def\IAI{{\ttt10}\alpha_2}  
\def\IAO{{\ttt11}\alpha{_2}}
\def\IIA{{\ttt10}\alpha{_1}}
\def\AIO{{\ttt01}\mu{_2}}
\def\OAI{{\ttt20} \alpha{_2}} 
\def\AOI{\mu{_2}{\ttt10}} 
\def\OOA{{\ttt30}\alpha{_1}}
%STENY
\def\BB{}
\def\NN{{\ttt02}\psi}
\def\WW{\phi_W}
\def\SS{{\ttt10}\psi}
\def\EE{\psi}
\[
  \xymatrix@R=1.2em@C=1.7em{
    \OOO \ar[rr]^{\OAO}\ar[dd]_{\AOO}
    &&\OIO \ar[dd]^{\AIO} \ar[rd]^{\OIA}
    %\ar@{=>}[lldd]|{\BB}
    \\
    &&& \OII  \ar[dd]^{\AII}
   % \ar@{=>}[ld]|{\EE\rule{.1em}{0pt}}
    \\
    \IOO\ar[rr]^{\IAO} \ar[rd]_{\IOA} && \IIO
    \ar@{=>}[ld]|{\SS\rule{.1em}{0pt}}
    \ar[dr]^{\IIA}
    \\
    &\IOI \ar[rr]^{\IAI} && \III
  }
  \hskip .1em
  \raisebox{-3em}{$=$} \hskip .1em
  \xymatrix@R=1.2em@C=1.7em{
    \OOO \ar[rr]^{\OAO}\ar[dd]_{\AOO} 
    \ar[rd]^(.6){\OOA}   && 
    \OIO \ar[rd]^{\OIA} 
\ar@{=>}[ld]|\NN
    \\
    &\OOI \ar[rr]^{\OAI}   \ar[dd]_{\AOI}
    %\ar@{=>}[ld]|{\WW\rule{.1em}{0pt}}
    && \OII 
    \ar[dd]^{\AII} 
  %  \ar@{=>}[lldd]|{\FF\rule{.1em}{0pt}}
    \\
    \IOO\ar[rd]_{\IOA} &&
    \\
    &\IOI \ar[rr]^{\IAI} && \III
  }
\]
which holds by the argument used in E.

\vskip .5em
\noindent 
G. The commutativity of the bottom front left cube
means the equality
%VRCHOLY
\def\OOO{{\ttt31}A}      
\def\OIO{{\ttt30}A}     
\def\OII{{\ttt20}A}      
\def\III{{\ttt10}A}        
\def\IOO{{\ttt21}A}      
\def\IOI{{\ttt11}A}    
\def\IIO{{\ttt20}A}         
\def\OOI{{\ttt21}A}     
%HRANY
\def\OAO{{\ttt30}\alpha_1}  
\def\OIA{{\ttt10}\mu_2}  
\def\AII{\mu{_2}} 
\def\AOO{{\ttt01}\mu{_2}{\ttt10}}  
\def\IOA{{\ttt01}\mu{_2}} 
\def\IAI{{\ttt10}\alpha_1}  
\def\IAO{{\ttt20}\alpha_1}  
\def\IIA{\mu{_2}}
\def\AIO{\mu{_2}{\ttt10}}
\def\OAI{{\ttt20} \alpha{_1}} 
\def\AOI{{\ttt11}\mu{_2}} 
\def\OOA{{\ttt11}\mu{_2}}
%STENY
\def\BB{}
\def\NN{??????}
\def\WW{\phi_W}
\def\FF{\TT1{}(\psi)}
\def\SS{???}
\def\EE{\psi}
\[
  \xymatrix@R=1.2em@C=1.7em{
    \OOO \ar[rr]^{\OAO}\ar[dd]_{\AOO}
    &&\OIO \ar[dd]^{\AIO} \ar[rd]^{\OIA}
    %\ar@{=>}[lldd]|{\BB}
    \\
    &&& \OII  \ar[dd]^{\AII}
   % \ar@{=>}[ld]|{\EE\rule{.1em}{0pt}}
    \\
    \IOO\ar[rr]^{\IAO} \ar[rd]_{\IOA} && \IIO
  %  \ar@{=>}[ld]|{\SS\rule{.1em}{0pt}}
    \ar[dr]^{\IIA}
    \\
    &\IOI \ar[rr]^{\IAI} && \III
  }
  \hskip -.1em
  \raisebox{-3em}{$=$} \hskip .1em
  \xymatrix@R=1.2em@C=1.7em{
    \OOO \ar[rr]^{\OAO}\ar[dd]_{\AOO} 
    \ar[rd]^(.6){\OOA}   && 
    \OIO \ar[rd]^{\OIA} 
%\ar@{=>}[ld]|\NN
    \\
    &\OOI \ar[rr]^{\OAI}   \ar[dd]_{\AOI}
    %\ar@{=>}[ld]|{\WW\rule{.1em}{0pt}}
    && \OII 
    \ar[dd]^{\AII} 
  %  \ar@{=>}[lldd]|{\FF\rule{.1em}{0pt}}
    \\
    \IOO\ar[rd]_{\IOA} &&
    \\
    &\IOI \ar[rr]^{\IAI} && \III
  }
\]
which is obvious, since all faces commute by naturality.

\vskip .5em
\noindent 
H. Finally, the bottom right cube is exactly the coherence cube for
$\alpha_2$, hence it commutes. 

This establishes the  equality in (ii) of Definition~\ref{definition:lax_alg}.
We still need to check the equalities in (i), (iii) and (iv). 
The verification of (i) is a straightforward use of definitions and
axioms, and we leave it to the reader. Let us turn to~(iii).

Recall that the unit for the composite 
monad is $\eta_{21}=\T_2\eta_1\cdot \eta_2$. The 2-cell 
$\phi_{21}\cdot\T_{21}\eta_{21}$ appearing in (ii) 
is the vertical composite of the following three 2-cells,
 each of which is in fact the identity 2-cell, to wit:
\begin{itemize}[topsep=0pt, partopsep=0pt, itemsep=0pt, itemindent=-1em]
\item[$\bullet$] 
$\phi_2\cdot\T_2^2\alpha_1\cdot\T_2^2\mu_1\cdot\T_2^2\T_1\eta_1\cdot\T_2\T_1\eta_2
=\phi_2\cdot\T_2^2\alpha_1\cdot\T_2\T_1\eta_2=
\phi_2\cdot\T_2\eta_2\cdot\T_2\alpha_1=\id,$
\item[$\bullet$] 
$\alpha_2\cdot\T_2\alpha_2\cdot\T_2^2\phi_1\cdot\T_2^2\T_1\eta_1\cdot\T_2\T_1\eta_2
=\alpha_2\cdot\T_2\alpha_2\cdot\id\cdot\T_2\T_1\eta_2=\id,$ and
\item[$\bullet$] 
$\alpha_2\cdot\T_2\psi\cdot\T_2^2\T_1\alpha_1\cdot\T_2^2\T_1\eta_1\cdot\T_2\T_1\eta_2=\alpha_2\cdot\T_2\psi\cdot\T_2\T_1\eta_2=\id.$ 
\end{itemize}
The proof of the equality $\phi_{21}\cdot\eta_{21}\T_{21}=\id_{\alpha}$ in
(iv) is similar.

\section{Lifted monads}

In this part of the Appendix we prove the existence of the
monads $\widetilde{\T}_i$, $i=1,2$, featured in
Proposition~\ref{proposition:lax_lax_characterization}.
\begin{proposition}
\label{proposition:lifted_monad}
The functor
$\widetilde{\T}_2 = (\T_2 A,\T_2\alpha_1\cdot\Lambda,
\T_2\phi_1\cdot\Lambda)$ defined in~\eqref{equation:lifted_T_2_lax_lax}
extends to a 2-monad on the 2-category
$\tt{Lax}\hbox{-}\T_{\!1}\tt{\hbox{-}alg}_{\tt{lax}}$ of Definition~\ref{definition:Lax-T-alg_lax}. The multiplication
has components
\[
\widetilde{\mu}_2{(A,{\alpha_1},\phi)}=(\mu_2(A),\id)\colon
\widetilde{\T}^2_2(A,{\alpha_1},\phi_1)\longrightarrow
\widetilde{\T}_2(A,{\alpha_1},\phi_1),
\]
and the unit has components
\[
\widetilde{\eta}_2{(A,{\alpha_1},\phi_1)}=(\eta_2(A),\id)\colon
(A,{\alpha_1},\phi_1) \longrightarrow \widetilde{\T}_2(A,{\alpha_1},\phi_1).
\]
An obvious version of the proposition holds also for $\widetilde \T_1$
in place of\/ $\widetilde \T_2$, and colax instead of lax morphisms.
\end{proposition}

\begin{proof} The concrete form of the isomorphism $\Lambda$ in formula~\eqref{equation:lifted_T_2_lax_lax} plays a role, so we rewrite the formula
 more explicitly.
The 2-functor $\widetilde{\T}_2$ on lax $\T_1$ algebras, their lax morphisms and algebra 2-cells
is given as
\begin{align*}
\widetilde{\T}_2(A,
  \alpha_1,\phi_1)&:=(\T_2 A,\T_2\alpha_1\cdot\lambda,
  \T_2\phi_1\cdot\T_2\lambda\cdot\lambda\T_2),
\\
\widetilde{\T}_2(f,\psi)&:=(\T_2f,
  \T_2\psi \cdot \lambda), \  \hbox { and}
\\
\widetilde{\T}_2\rho&:=\T_2\rho.
\end{align*}
We prove that the components of the multiplication
$\widetilde{\mu}_2$ are lax $\T_1$-algebra morphisms by verfying 
equations (i) and (ii) of
Definition~\ref{definition:lax_alg_lax_morph}. Equation (ii)
is is easy to verify. Equation (i) requires that the composed two 2-cells
in Figures~\ref{Vcera byl ten koncert prepaleny az do prahu bolesti.I}
and~\ref{Vcera byl ten koncert prepaleny az do prahu bolesti.II}
are equal. 
\begin{figure}
\[
\xymatrix@C=.9cm{
{\iii22}A \ar[rrr]^{\iii20 \mu_2}  \ar[ddd]_{\mu_1\iii02} 
&&&
{\iii21}A \ar[ddd]_{\mu_1\iii01} \ar[rd]^{\iii10\lambda}
\\ 
&&&& {\iii10\iii01\iii10} A \ar[d]_{\lambda \iii10} \ar[rrdd]^{\iii11 \alpha_1} 
\\ 
&&&&{\iii01\iii20}A \ar[dd]^{\iii01 \mu_1}  \ar[rrdd]^{\iii01\iii10\alpha_1}
\\
{\iii12}A \ar[rrr]^{\iii10 \mu_1} \ar[rd]^{\lambda\iii01}   &&& {\iii11}A \ar[dr]^\lambda  &&&
{\iii10\iii01}A \ar[d]^\lambda
\\
& {\iii01\iii10\iii01}A \ar[dr]^{\iii01 \lambda}   &&& {\iii01\iii10}A  \ar[ddrr]^{\iii01\alpha_1}   && {\iii01\iii10}A
\ar[dd]^{\iii01\alpha_1} \ar@{=>}[ll]|{\iii01\phi_1}
\\
&& {\iii02\iii10}A \ar@/^2em/[rru]^{\mu_2 \iii10} \ar[dr]^{\iii02\alpha_1}
\\
&&& {\iii02}A \ar[rrr]^{\mu_2}   &&\hphantom{LLL}& {\iii01}A
}
\]
\caption{\label{Vcera byl ten koncert prepaleny az do prahu bolesti.I}
A $2$-cell $\iii01\alpha_1 \cdot \lambda \cdot
  \iii11\alpha_1 \cdot \iii10 \lambda \cdot \iii20\mu_2
\Longrightarrow \mu_2 \cdot \iii02\alpha_1 \cdot \iii01 \lambda \cdot
\lambda \iii01 \cdot \mu_1 \iii02$.} 
\end{figure}
\begin{figure}
\[
\xymatrix@C=.9em{
{\iii22}A \ar[rrr]^{\iii20 \mu_2} \ar[rd]^{\iii10\lambda \iii01}  
\ar[ddd]_{\mu_1\iii02}
&&& {\iii21}A \ar[rd]^{\iii10\lambda}
\\
&{\iii10\iii01\iii10\iii01}A \ar[d]_{\lambda\iii10\iii01}
\ar[dr]^{\iii10\iii01\lambda} &&& {\iii10\iii01\iii10}A
\ar[ddrr]^{\iii11\alpha_1}  \ar@{-->}[d]_{\lambda \iii10}
\\
& {\iii01\iii20\iii01}A \ar[dr]^(.4){\iii01\iii10\lambda} 
\ar[dd]_(.4){\iii01\mu_1\iii01}  & {\iii10\iii02\iii10}A
\ar@/^2em/[rru]^{\iii10\mu_2\iii10} \ar[d]^(.6){\lambda\iii01\iii10} 
\ar[rd]^{\iii12\alpha_1}&&{\ttt12}A
\\
{\iii12}A  \ar[rd]^{\lambda \iii01}  && {\iii01\iii10\iii01\iii10}A
\ar[dr]^(.4){\iii01\iii11\alpha_1} \ar[d]_{\iii01\lambda\iii10}
 &{\iii12}A
\ar[rrr]^{\iii10\mu_2}  
\ar[d]^{\lambda \iii01}
&&&
{\iii11}A \ar[d]^\lambda
\\
&{\iii01\iii11}A \ar[rd]^(.4){\iii01\lambda} & {\iii02\iii20}A 
\ar[d]_{\iii02\mu_1}  \ar[dr]^{\ttt21\alpha_1}
\ar@/^2.5em/@{-->}[rruu]^(.75){\mu_2 \iii20}|(.28){\phantom{\rule{1em}{1em}}}
& {\iii01\iii11}A
\ar[d]^{\iii01\lambda} &&& {\ttt11}A \ar[dd]_{\iii01\alpha_1}
\\
&&{\iii02\iii10}A    \ar[rd]^(.6){\iii02\alpha_1}   
& {\iii02\iii10}A \ar@/^2em/[rrru]^{\mu_2 \iii10} \ar[d]^{\iii02\alpha_1} 
\ar@{=>}[l]|(.465){\iii02\phi_1}
\\
&&& {\iii02}A \ar[rrr]^{\mu_2} &&\hphantom{LLLLLLLLL}& {\iii01}A
}
\]
\caption{\label{Vcera byl ten koncert prepaleny az do prahu bolesti.II}
A $2$-cell $\iii01\alpha_1 \cdot \lambda \cdot
  \iii11\alpha_1 \cdot \iii10 \lambda \cdot \iii20\mu_2
\Longrightarrow \mu_2 \cdot \iii02\alpha_1 \cdot \iii01 \lambda \cdot
\lambda \iii01 \cdot \mu_1 \iii02$.} 
\end{figure}
To this end, notice that the diagram 
\[
\xymatrix{
{\iii22} \ar[r]^(.4){\iii10\lambda\iii01}\ar[d]_{\iii20\mu_2}   &
\ar[r]^{\iii11 \lambda} {\iii11\iii11} &
\ar[r]^{\lambda\ttt11} {\iii12\ttt01} \ar[d]^{\iii10 \mu_2 \iii10} & 
\ar[r]^(.6){\iii01\lambda\iii10} {\ttt11\ttt11}& {\iii02\iii20}
\ar[d]^{\mu_2\iii20}
\\
{\iii21} \ar[rr]^{\iii10\lambda} &&
 \ar[rr]^{\lambda\iii10} {\iii11\iii10} && {\ttt12}
}
\]
embedded to the coherence cube as indicated in Figure~\ref{Vcera byl
  ten koncert prepaleny az do prahu bolesti.II}, 
commutes. The equality of the two $2$-cells in 
Figures~\ref{Vcera byl ten koncert prepaleny az do prahu bolesti.I}
and~\ref{Vcera byl ten koncert prepaleny az do prahu bolesti.II}
follows from the $2$-naturality of
$\mu_2$ and the diagram
\[
\xymatrix@C=-.4em{
&{\ttt22} A\ar[rrrrrrrrrrr]^{\mu_2\ttt02}
 \ar@/^2.1em/[dd] \ar@/_2,1em/[dd] &&&&&&&\hphantom{BBBBBB}&&&&  
\ar@/_2.1em/[dd]  \ar@/^2.1em/[dd]   {\ttt12}  A
\\
\ar@{=>}[rr]|{\ttt20\phi_1} &&&&&&&&&&&\ar@{=>}[rr]|{\ttt10\phi_1}&&
\\
&{\ttt20} A \ar[rrrrrrrrrrr]^{\mu_2} &&&&&&&&&&& {\ttt10} A.
}
\]
Notice that the same argument was used to establish the
commutativity of the cubes C, E and F in the proof of
Theorem~\ref{Ozivil jsem Jarcino auto.}. 
 
The verification 
that the components of $\widetilde{\eta}_2$ are algebra morphisms is analogous.
Clearly, $\widetilde{\mu}_2$ and $\widetilde{\eta}_2$ satisfy the classical monad axioms, since they are defined by $\mu_2$ and $\eta_2$.
\end{proof}

\listoftodos

\bibliographystyle{plain} 

\end{document}